\documentclass[times]{cuparticle}

\volume{}
\doi{}

\usepackage{tabulary,amsmath,amssymb,amsfonts}
\usepackage{amsthm}
\usepackage{bbm}
\usepackage{comment}

\usepackage[T1]{fontenc}
\usepackage[utf8x]{inputenc}

\newcommand{\IN}{\mathbb{N}}

\newcommand{\IZ}{\mathbb{Z}}
\newcommand{\IR}{\mathbb{R}}

\newcommand{\N}{\mathbb{N}}

\newcommand{\Z}{\mathbb{Z}}
\newcommand{\R}{\mathbb{R}}

\newcommand{\dif}{\ \mathrm{d}}

\newcommand{\sF}{\mathcal{F}}

\newcommand{\sH}{\mathcal{H}}

\newcommand{\reff}[1]{(\ref{#1})}

\newcommand{\IP}{\mathbb{P}}
\newcommand{\IE}{\mathbb{E}}
\newcommand{\Ii}{\mathbbm{1}}

\newcommand{\Var}{\mbox{Var}}

\newcommand{\Cov}{\mathrm{Cov}}
\newcommand{\tr}{\mathrm{tr}}

\newcommand{\dto}{\overset{d}{\to}}

\newcommand{\sK}{\mathcal{K}}

\newtheorem{theorem}{Theorem}[section]
\newtheorem{corollary}[theorem]{Corollary}
\newtheorem{definition}[theorem]{Definition}
\newtheorem{proposition}[theorem]{Proposition}
\newtheorem{lemma}[theorem]{Lemma}
\newtheorem{remark}[theorem]{Remark}

\newtheorem{assumption}[theorem]{Assumption}

\DeclareMathOperator{\argmin}{argmin}

\begin{document}

\authorheadline{Rainer Dahlhaus and Stefan Richter}
\runningtitle{Adaptation for locally stationary processes}

\begin{frontmatter}

\title{Adaptation for nonparametric estimators of locally stationary processes}

\author[1]{Rainer Dahlhaus}
 
\address[1]{Institut f{\"{u}}r Angewandte Mathematik\unskip, 
    Heidelberg University\unskip, Im Neuenheimer Feld 205\unskip, Heidelberg\unskip, Germany
  %\ead{dahlhaus@statlab.uni-heidelberg.de}
  }

\author[1]{Stefan Richter}
 
\received{2019}

\begin{abstract}

Two adaptive bandwidth selection methods for nonparametric estimators in locally stationary processes are proposed. We investigate a cross validation approach and a method based on contrast minimization and derive asymptotic properties of both methods.
The results are applicable for different statistics under a broad setting of locally stationarity including nonlinear processes. At the same time we deepen the general framework for local stationarity based on stationary approximations. For example a general Bernstein inequality is derived for such processes. A simulation study performed on the covariance function and more complicated functionals shows that both adaptation methods work  well.

\end{abstract}

%\MSC[2010]{16W10 (primary);  16D50 (secondary)}

\end{frontmatter}

% \linenumbers

%\begin{center}
%	Keywords: Non-stationary processes, Minimax-optimal, Model selection
%\end{center}

%\tableofcontents

\section{Introduction}

In this paper we develop data adaptive bandwidth selection rules for nonstationary processes under a novel paradigm of local stationarity recently introduced in \cite{richterdahlhaus2018bernoulli}. Time series sampled at high frequency or just long time series exhibit more frequently nonstationarity instead of stationarity, and correspondingly the use of models with time varying parameters or of locally stationary processes in general has increased a lot during recent years. A prominent example from financial econometrics is the use of GARCH-models to model conditional heteroscedasticity: While in the beginning ordinary GARCH-models have been regarded as sufficient to model conditional heteroscedasticity of the volatility, insight has grown that for example modeling of the daily pattern can be improved by a time varying GARCH-model (cf.
\cite{AmadoTeraesvirta2013je}, \cite{AmadoTeraesvirta2017er},\cite{DahlhausSubbaRao2006}). Analogously, time varying models for the trading intensity could be used such as locally stationary Hawkes models (\cite{vonSachsRoueff2016spa}). Also motivated by financial returns Koo and Linton \cite{KooLinton2012je} have studied locally stationary diffusion processes with a time varying drift and a volatility coefficient varying over time and space.

The intention and novelty of this paper is twofold. On the one hand we want to establish methods for adaptive bandwidth selection for nonparametric estimators, on the other hand we want to deepen  the general framework for local stationarity based on stationary approximations. The latter goes hand in hand with the former since several results (such as Bernstein inequalities) are of high value beyond the topic of adaptation. The development of such a general framework / hyper-model for locally stationary processes is important, since such a hyper-model can for example serve as a general assumption for nonparametric estimation, as a framework to prove general technical results such as strong laws of large numbers or the Bernstein inequality of this paper, as a framework to judge parametric models under model-misspecification, or as a setting for model selection strategies. General frameworks for locally stationary processes that have been used before, are time varying linear processes as in \cite{Dahlhaus1997},
and time-varying Bernoulli shifts in combination with the functional dependence measure (\cite{Wu2005},\cite{WuAndZhou2011}). Furthermore, processes with evolutionary spectra in the setting of Priestley (\cite{Priestley1965}, \cite{Priestley1988})  may also be regarded as a hyper-model - although the setting does not allow for asymptotic considerations in a strict mathematical sense since it is not an infill asymptotics approach.

In \cite{richterdahlhaus2018bernoulli} we have introduced a different framework which formalizes the original idea behind local stationarity - namely that at each point in time the observed nonstationary process can be approximated by a stationary process. Such a property was proved in the context of time varying ARCH-processes in \cite{DahlhausSubbaRao2006} and investigated further in the context of random coefficient models in \cite{SubbaRao2006}. The use of such approximations as a general model was recommended by \cite{vogt2012}, who investigated
nonparametric regression for locally stationary time series, and by \cite{KooLinton2012je}, who investigated semiparametric estimation for locally stationary models.

Within this new framework, the focus of this paper is on deriving methods for adaptive bandwidth selection of general nonparametric estimators. These include for example estimators of the time varying covariance function, the autocorrelation function, the time varying characteristic function or general moment estimators.
%\textbf{Stefan}
Different to nonparametric regression, there exist only very few theoretical results about adaptivity for locally stationary processes. We mention \cite{mallat1998} who discussed adaptive covariance estimation for a general class of locally stationary processes. Other results are constructed for specific models and are partly dependent on further tuning parameters: \cite{giraud2015} discussed online-adaptive forecasting of tvAR processes and \cite{arkoun2008}, \cite{arkoun2011} proposed methods for sequential and minimax-optimal bandwidth selection for tvAR processes of order 1. In \cite{richterdahlhaus2018aos} adaptive estimation was developed for 
time varying parameter curves by means of local M-estimators (i.e. in a locally parametric setting), while in this paper the task is nonparametrical inference also locally. Technically, the difference is that we do no longer assume that the observed time series comes from a specific model like tvGARCH or tvAR (and use this knowledge to build the estimator). Instead we are interested in general properties of the time series like the mean, covariances, correlations or characteristic functions.
%\textbf{Stefan}
 
In Section \ref{sec2} we introduce the framework of local stationarity based on stationary approximations and derivative processes. We give a short overview of the basic results from \cite{richterdahlhaus2018bernoulli} and extend these results in that we prove an invariance property of the results also for nonlinear transformations based on infinitely many lags. In Section \ref{sec3} we prove asymptotic optimality of a global bandwidth selection approach based on cross validation  with respect to a mean squared error type distance measure. We  discuss its behavior in practice via simulations. In Section \ref{sec4} we investigate a local bandwidth selection procedure using a contrast minimization approach in the spirit of \cite{lepski2011}. We prove that the resulting nonparametric estimator attains the optimal rate for the mean squared error up to a log-factor. We compare the obtained method with a global optimal selection routine and show the superiority of our method in selected examples. The section contains also a Bernstein inequality which is of interest beyond the present paper. Section \ref{sec5} contains some conclusions. The Appendix in Section \ref{sec6} contains several technical results including the proofs of the main theorems and a more general result for the setting in Section \ref{sec4}.

\section{The Model and Main results}
\label{sec2}

\subsection{The Model}

We assume that we observe $n$ realizations of a process $X_{t,n}$ at time points $t = 1,...,n$. 
The process is considered to be locally stationary in the following sense (cf. \cite{richterdahlhaus2018bernoulli}):
\begin{assumption}
    \label{def_localstat}
	Let $q \ge 1$. There exists some $D > 0$ and for each $u\in[0,1]$, there exists a (strictly) stationary process $(\tilde X_t(u))_{t\in\IZ}$ such that for all $t = 1,...,n$ and $u,u'\in[0,1]$, $\sup_{u\in[0,1]}\|\tilde X_0(u)\|_q \le D$, $\sup_{t,n}\|X_{t,n}\|_q \le D$ and
	\[
		\|X_{t,n} - \tilde X_t(t/n)\|_q \le D n^{-1}, \quad\quad \|\tilde X_0(u) - \tilde X_0(u')\|_q \le D |u-u'|.
	\]
	Here, we use $\|Z\|_q := \IE[|Z|^q]^{1/q}$ for random variables $Z$.
\end{assumption}
The conditions mean that $X_{t,n}$ can be approximated locally, for $|u-\frac{t}{n}|  \ll 1$, by a stationary process $\tilde X_t(u)$. The continuity condition stated on $u\mapsto \tilde X_t(u)$ implies that the stationary approximations vary smoothly over time. This motivates the interpretation of locally stationary processes as processes which change there (approximate) stationary properties smoothly over time. The main properties of $X_{t,n}$ therefore are encoded in the stationary approximations and it is therefore of interest to analyze terms of the form $\IE g(\tilde X_t(u),\tilde X_{t-1}(u),...)$ which are a natural approximation of $\IE g(Y_{t,n},Y_{t-1,n},...)$.\\

More detailed, define $Y_{t,n} := (X_{t,n},X_{t-1,n},...,X_{1,n},0,0,...)$ and $\tilde Y_t(u) := (\tilde X_{s}(u):s \le t)$. Our goal is to estimate functionals of the form
\[
	u \mapsto G(u) := \IE g(\tilde Y_{t}(u)),
\]
where $g:\R^{\N} \to \R^d$ is some measurable function. Important examples are:
\begin{itemize}
	\item Time-varying covariances, $c(u,k) := \IE \tilde X_0(u) \tilde X_k(u)$ with $g(x_0,x_1,...,x_k) := x_0 x_k$,
	\item Time-varying characteristic functions,  $\phi(t,u) := \IE e^{i t \tilde X_0(u)}$, with  $g_t(x) = e^{itx}$.
	\item Time-varying distribution functions, $F(y,u) := \IE \Ii_{\{\tilde X_t(u) \le y\}}$, with $g_y(x) := \Ii_{\{x\le y\}}$.
	%\item Spectral distribution function ???
\end{itemize} 
A standard estimator is given by a localized moment estimator,
\begin{equation}
	\hat G_h(u) := \frac{1}{n}\sum_{t=1}^{n}K_h(t/n-u)\cdot g(Y_{t,n}),\label{model_estimator}
\end{equation}

where $h\in (0,\infty)$ is some bandwidth and $K$ is a kernel function coming from the class $\sK$ defined below.

\begin{definition}
	A function $K$ is in the set $\sK$ if $K$ is nonnegative, Lipschitz continuous, has support $[-\frac{1}{2},\frac{1}{2}]$ and $\int K(x)\dif x = 1$. We set $|K|_{\infty} := \sup_{x\in[-\frac{1}{2},\frac{1}{2}]}|K(x)|$.
\end{definition}

In the following we present a general theory how to obtain asymptotic results for such estimators with a focus on adaptation, i.e. on choosing the bandwidth $h$. We therefore assume that $g$ belongs to the class $\sH(M,\chi,C)$ below which is a Lipschitz-type condition with polynomially growing constants (ie. the Lipschitz condition is relaxed for larger $x$ and $y$). For some sequence of non-negative real-valued numbers $\chi = (\chi_i)_{i\in\IN}$ and some sequence of complex-valued numbers $x = (x_i)_{i\in\IN}$ define $|x|_{\chi} := \sum_{i\in\N}\chi_i |x_i|$.

\begin{definition}
	We say that $g:\IR^{\IN} \to \IR$ belongs to the class $\sH(M,\chi,C)$ if there exists some $M\in\IN$, some constant $C > 0,\varepsilon > 0$ and some sequence of nonnegative real numbers $\chi = (\chi)_{i\in\IN}$ with $\chi_i = O(i^{-2-\varepsilon})$ such that
	\[
		\sup_{x\not=y}\frac{|g(x) - g(y)|}{|x-y|_{\chi}\cdot (1 + |x|_{\chi}^{M-1} + |y|_{\chi}^{M-1})}\le C.
	\]
	A function $g:\IR^{\IN}\to \IR^d$ ($d\in\IN$) is in $\sH(M,\chi,C)$ if each component belongs to $\sH(M,\chi,C)$.
\end{definition}

By Hoelder's inequality it is easy to see that the following 'invariance principle' of local stationarity holds:
\begin{proposition}
	If $(X_{t,n})_{t=1,...,n}$ is locally stationary in the sense of Assumption \ref{def_localstat} (with  some $q > 0$) and $g \in \sH(M,\chi,C)$, then the same holds for $g(Y_{t,n})$ (with $q' = \frac{q}{M}$).
\end{proposition}

Based on this result we can use Theorem 2.7 in \cite{richterdahlhaus2018bernoulli} to obtain

\begin{theorem}[Consistency] If $X_{t,n}$ is locally stationary in the sense of Assumption  \ref{def_localstat} with $q \ge M$ and $g \in \sH(M,\chi,C)$, then
\[
	\hat G_h(u) \to G(u)
\]
in probability and in $L^1$ provided that $nh\to\infty$, $h\to 0$.
\end{theorem}

To provide a more detailed expansion of the bias $\IE \hat G_h(u)$, we need the following additional assumption on the stationary approximation sequence:

\begin{assumption}\label{assumption_derivative_process}
	For each $t\in\Z$, the process $u \mapsto \tilde X_t(u)$ is twice continuously differentiable with
	\[
	    \|\sup_{u\in[0,1]}|\partial_u \tilde X_t(u)|\|_M \le D, \quad\quad \|\sup_{u\in[0,1]}|\partial_u^2 \tilde X_t(u)|\|_M \le D
	\]
	Additionally, $g \in \sH(M,\chi,C)$ is twice continuously differentiable such that for all $i,j$, $\partial_{x_i} g \in \sH(\max\{M-1,1\},\chi,C \chi_i)$ and $\partial_{x_{i}x_{j}}^2 g \in \sH(\max\{M-2,1\},\chi,C \chi_i \chi_j)$.
\end{assumption}

The following theorem is a combination of Lemma 3.3  from the supplementary material of \cite{richterdahlhaus2018aos} and Lemmas \ref{lemma_statapprox}, \ref{lemma_variance_calc} from the Appendix.

\begin{theorem}[Bias expansion and MSE decomposition]\label{theorem_bias} Suppose that Assumption \ref{def_localstat} and Assumption \ref{assumption_derivative_process} are fulfilled for $q = 2M$. Let $K\in \sK$.

Define $\sigma_K^2 = \int K(x)^2 \dif x$, $\mu_K := \int K(x) x^2 \dif x$ and the long-run variance of $\tilde X_t(u)$ (see Theorem \ref{thm_asymptotic_normality} below),
\[
	\Sigma(u) = \sum_{k\in\IZ}\Cov(g(\tilde Y_0(u)),g(\tilde Y_k(u))).
\]
Then $u \mapsto G(u)$ is twice continuously differentiable and uniformly in $u \in [-\frac{h}{2},\frac{h}{2}]$, it holds that
\begin{enumerate}
    \item[(i)] \begin{equation}
	|\IE \hat G_h(u) - G(u) - \frac{h^2}{2}\mu_K \cdot \partial_u^2 G(u)|_2 = O((nh)^{-1}) + o(h^2),\label{formel_bias}
\end{equation}
    \item[(ii)] \begin{equation}
	\IE|\hat G_h(u) - G(u)|_2^2 = \frac{\sigma_K^2}{nh}\cdot \tr(\Sigma(u)) + \frac{h^4}{4}\mu_K^2 \cdot |\partial_u^2 G(u)|_2^2 +  o((nh)^{-1} + h^4).\label{formel_mse}
	\end{equation}
\end{enumerate}
\end{theorem}

From the decomposition (ii) in Theorem \ref{theorem_bias} it can be easily seen that $h \mapsto \IE |\hat G_h(u) - G(u)|_2^2$ is minimized by
\begin{equation}
    h_{opt}(u) = \Big(\frac{4 \sigma_K^2 \tr(\Sigma(u))}{\mu_K^2|\partial_u^2 G(u)|_2^2}\Big)^{1/5}\cdot n^{-1/5}\label{hopt_lokal}
\end{equation}
if $|\partial_u^2 G(u)|_2^2 > 0$. Accordingly, we obtain that $h \mapsto \int_0^{1}\IE|\hat G_h(u) - G(u)|_2^2\, w(u)\dif  u$ is minimized by
\begin{equation}
    h_{opt} = \Big(\frac{4 \sigma_K^2 \int_0^{1}\tr(\Sigma(u)) w(u) \dif u}{\mu_K^2\int_0^{1}|\partial_u^2 G(u)|_2^2 w(u) \dif u}\Big)^{1/5}\cdot n^{-1/5}\label{hopt_global}
\end{equation}
where $w:[0,1] \to [0,\infty)$ is some weight function taking care of boundary issues.

\subsection{Asymptotic normality}

To prove central limit theorems, we additionally have to assume mixing conditions. In the setting of Assumption \ref{def_localstat} it is enough to state these assumptions pointwise on the stationary approximations $\tilde X_t(u)$. An elegant way to formulate mixing assumptions was the introduction of the functional dependence measure of \cite{Wu2005}. Suppose that $\zeta_t$, $t\in\Z$ is a sequence of i.i.d. random variables, and put $\sF_t := (\zeta_t,\zeta_{t-1},...)$, $t \ge 0$. Let $\zeta_t^{*}$, $t\in \Z$ be an independent copy of $\zeta_t$, $t\in\Z$, and put $\sF_t^{*} := (\zeta_t,\zeta_{t-1},...,\zeta_1,\zeta_0^{*},\zeta_{-1},\zeta_{-2},...)$, $t \ge 0$.

\begin{assumption}\label{assumption_mixing_clt}
    Assume that for each $u\in[0,1]$, $\tilde X_t(u) = H(\sF_t,u)$ with some measurable function $H$. Suppose that
    \[
        \delta_q^{\tilde X(u)}(k) := \|\tilde X_t(u) - \tilde X_t^{*}(u)\|_q, \quad k \ge 0,
    \]
    fulfills $\sum_{k=0}^{\infty}\sup_{u\in[0,1]}\delta_q^{\tilde X(u)}(k) < \infty$.
\end{assumption}

The following theorem is due to Theorem 2.10 in \cite{richterdahlhaus2018bernoulli} and Theorem \ref{theorem_bias}:

\begin{theorem}[Asymptotic normality]\label{thm_asymptotic_normality} Suppose that Assumption \ref{def_localstat}, Assumption \ref{assumption_derivative_process} and Assumption \ref{assumption_mixing_clt} are fulfilled for $q = 2M$. Additionally, assume that $\|\sup_{u\in[0,1]}|\tilde X_t(u)|\|_{2M} < \infty$. Then, as $n\to\infty$, $h\to 0$ and $nh^5 = O(1)$,
\[
	\sqrt{nh}\big[\hat G_h(u) - G(u) - \frac{h^2}{2}\mu_K\cdot \partial_u^2 G(u)\big] \dto N(0, \sigma_K^2\cdot \Sigma(u)).
\]
\end{theorem}

\section{Global adaptive bandwidth selection: Cross Validation}
\label{sec3}

In the following we discuss a global bandwidth selection method for $\hat G_h(u)$. The goal is to find an estimator which minimizes
\[
	d_M(h) := \int_0^{1}|G(u) - \hat G_h(u)|_2^2 \, w(u) \dif u,
\]
where $w:[0,1] \to [0,\infty)$ is some weight function of bounded variation which takes care of the boundary effects. Typically, $w(\cdot) = \Ii_{[\gamma,1-\gamma]}(\cdot)$ with some $\gamma>0$.  
In practice this means to find an estimator $\hat h$ which is close to $h_{opt}$ from \reff{hopt_global}.

An important advantage of cross validation over other methods is that it allows bandwidth selection without introducing sensitive tuning parameters. Typically, cross validation works well for global bandwidth selection but gets instable due to its high variance if one wants to use it locally. We therefore only present a theory for global selection. The next chapter discusses local bandwidth selection which aims to find an estimator for $h_{opt}(u)$ for each $u\in[0,1]$ separately.

%Cross validation is a tool to get a first estimate of $h_{opt}$. 
%can be used to get a 'first guess' of $h_{opt}$ which may be refined by using more sophisticated procedures  like plug-in estimation. 

%We present two approaches to select $h$ from the estimators given in \reff{model_estimator}. While the first method allows for global bandwidth selection (one $h$ for all $u\in[0,1]$) and is based on a cross validation, the second is based on a contrast minimization approach and allows for local bandwidth selection (one bandwidth $h = h(u)$ for each $u\in[0,1]$).

The method presented here is based on the following idea: $G(u) = \IE g(\tilde Y_0(u))$ is a minimizer of the functional
\[
	a \mapsto \IE\big[ \big(g(\tilde Y_0(u)) - a\big)^2\big].
\]
Therefore we expect $G$ to be a minimizer of the empirical version
\[
	H(G) := \frac{1}{n}\sum_{t=1}^{n}\big| g(Y_{t,n}) - G(t/n)\big|_2^2\, w(t/n).
\]
 It turns out that the simple plug-in approach to minimize $h \mapsto H(\hat G_h)$ does not work. The reason being that in this case, $H(\hat G_h)$ is no unbiased estimator of $\int_0^{1}|g(\tilde Y_t(u)) - \hat G_h(u)|_2^2\, w(u) \dif u$.
 
To obtain an unbiased estimator, we have to eliminate the dependencies which occur between the observation $g(Y_{t,n})$ and $\hat G_h(t/n)$ which leads to a natural change of the estimator $\hat G_h$:
Define for $\alpha \in (0,1)$, $\varepsilon > 0$,
\[
	K^{(n)}(x) := \begin{cases}
		K(x), & |x| \ge n^{-\alpha},\\
		0, & |x| \le (1-\varepsilon) n^{-\alpha},\\
		\text{linear}, & (1-\varepsilon) n^{-\alpha} < |x| < n^{-\alpha},
	\end{cases}
\]
as a Lipschitz-continuous approximation of $K(x)\Ii_{\{|x| \ge n^{-\alpha}\}}$, and
\[
	\hat G_h^{-}(u) := \Big(\frac{1}{n}\sum_{t=1}^{n}K_h^{(n)}(t/n-u)\Big)^{-1}\cdot \frac{1}{n}\sum_{t=1}^{n}K_h^{(n)}(t/n-u) \cdot g(Y_{t,n}).
\]
Note that $\hat G_h^{-}(t/n)$ and $g(Y_{t,n})$ are now (approximately) uncorrelated if the sequence $g(Y_{t,n})$, $t = 1,...,n$ fulfills appropriate dependence conditions.

We now define $\hat h$ accordingly to our original idea as
\begin{equation}
	\hat h := \argmin_{h\in H_n} H(\hat G_h^{-}).\label{def_crossval}
\end{equation}
The final estimator of $G$ is then given by $\hat G_{\hat h}$.

To judge the quality of $\hat G_{\hat h}$, we use the mean-squared error distance $d_M(h).$

%define the following distance measure which is a natural approximation of $\int |G(u) - \hat G_h(u)|_2^2 w(u) \dif u$:
%\[
%	d_A(h) := \frac{1}{n}\sum_{t=1}^{n}|G(t/n) - \hat G_h(t/n)|_2^2 w(t/n).
%\]
%\[
%	d_M(h) := \int_0^{1}|G(u) - \hat G_h(u)|_2^2 w(u) \dif u.
%\]

The following theorem states that $\hat h$, chosen by the cross validation procedure \reff{def_crossval}, is asymptotically optimal in the sense that $\hat h$ (i.e. the estimator $\hat G_{\hat h}$ associated to $\hat h$) attains the minimal distance to $G$ with respect to $d_A$ over all possible bandwidths $h\in H_n$.

\begin{theorem}[Asymptotic optimality of the bandwidth selector $\hat h$]\label{thm_crossval}$\,$\\
Suppose that Assumption \ref{def_localstat} holds for all $q > 0$ with some constants $D_q > 0$. Suppose that
$\sup_{u\in[0,1]}\delta^{\tilde X_0(u)}_q(k) = O(k^{-\kappa})$ with some $\kappa > 3$, and $g \in \sH(M,\chi,C)$ with some $\chi = (\chi_i)_{i\in\N}$ with $\chi_i = O(i^{-\kappa})$.\\
Let $K \in \sK$. For arbitrary small $\eta > 0$, let $H_n = [n^{-1+\alpha+\eta},n^{\min\{2\alpha-1,0\}-\eta}]$. Assume that the support of $w$ is $[\gamma,1-\gamma]$ with some $\gamma > 0$.
Then almost surely,
\[
	\lim_{n\to\infty}\frac{d_M(\hat h)}{\inf_{h\in H_n}d_M(h)} = 1.
\]
\end{theorem}
\begin{remark}
\begin{itemize}
	\item Choice of $\alpha$: From a theoretical point of view, it is possible to choose $\alpha \in (0,1)$ very near to 1. In practice however it may lead to more stable choices to take $\alpha$ smaller, for instance $\alpha \approx \frac{1}{2}$. This is discussed in more detail in the simulations below.
	\item Choice of $\varepsilon$: While it is necessary for theoretical proofs to choose $\varepsilon>0$ (so that $K^{(n)}$ is still Lipschitz continuous), in practice there seems to be no drawback when using $\varepsilon=0$.
	\item Choice of $H_n$: In practice it is not necessary to bound $H_n$ from above. However, to obtain meaningful bandwidth selections it is necessary to implement the lower bound given in the conditions of Theorem \ref{thm_crossval} or directly restrict the search to local minima.
\end{itemize}
\end{remark}

By using the representation \reff{formel_mse} integrated over $u \in [0,1]$, we directly obtain
\begin{corollary}
    Under the assumptions of Theorem \ref{thm_crossval}, almost surely
    \[
        \frac{\hat h}{h_{opt}} \to 1,
    \]
    where $h_{opt}$ is defined in \reff{hopt_global}.
\end{corollary}

\begin{remark}[Cross validation for compositions of moment estimators]\label{remark_correlation}
Many common estimators are given as compositions of moment estimators, a prominent example being the autocorrelation function of lag 1,
\[
    \gamma(u) := \frac{c(u,1)}{c(u,0)}, \quad\quad c(u,k) = \IE \tilde X_0(u) \tilde X_k(u).
\]
A natural way to estimate $\gamma(u)$ is to estimate $c(u,0)$, $c(u,1)$ by applying the above cross validation method and afterwards calculating $\hat \gamma(u)$ from the corresponding estimators $\hat c(u,0)$, $\hat c(u,1)$. It is obvious that this may not lead to a good result if the nature of $c(u,0)$ and $c(u,1)$ is very different from $\gamma(u)$ itself. Here, we give a brief sketch how to generalize our method to this case.

Generally speaking, we are interested in estimating $F(G(u))$ where  $F:\R^d \to \R^{\tilde d}$ is some given function. The obvious generalization of \reff{def_crossval} by defining $\tilde H(G) = \frac{1}{n}\sum_{t=1}^{n}\big| F(g(Y_{t,n})) - F(G(t/n))\big|_2^2 w(t/n)$ is not feasible due to exploding values of $F(g(Y_{t,n}))$. Instead, we approximate $F(g(Y_{t,n})) - F(G(t/n))$ by a Taylor expansion of order one and define
\begin{eqnarray}
    \tilde H(G) &:=&  \frac{1}{n}\sum_{t=1}^{n}\big|\partial F(G(t/n))\cdot (g(Y_{t,n}) - G(t/n))\big|_2^2 w(t/n)\nonumber\\
    &=& \frac{1}{n}\big|g(Y_{t,n}) - G(t/n)\big|_{F(G(t/n))^t F(G(t/n))}^2 w(t/n),\label{cv_modified}
\end{eqnarray}
where $|x|_A^2 := x^t A x$ for some vector $x$ and matrix $A$. We then define, as before,
\[
    \hat h^{comp} := \argmin_{h\in H_n}\tilde H(\hat G_h^{-}).
\]
From \reff{cv_modified} it can be seen that instead of finding a minimizer of $\int_0^{1}|\hat G_h(u) - G(u)|_2^2 w(u) \dif u$, we now minimize a weighted combination $\hat G_h(u) - G(u)$ (weighted by the matrix $F(G(u))^t F(G(u))$) which introduces the specific nature of the function $F$. In view of techniques from \cite{richterdahlhaus2018aos}, we conjecture that a similar result as in Theorem \ref{thm_crossval} holds for $\hat h^{comp}$ and a modified distance measure $d_M^{comp}(h):= \int_0^{1}|F(G(u)) - F(\hat G_h(u))|_2^2 w(u) \dif u$.
\end{remark}

\begin{remark}[Cross validation for functional moment estimators]\label{remark_characteristicfunction} Suppose that instead of $G(u) = \IE g(\tilde Y_t(u))$, we are interested in estimating $G_{\theta}(u) := \IE g_{\theta}(\tilde Y_t(u))$ uniformly in $\Theta$, where $g_{\theta}:\IR^{\N} \to \IR^d$ and $\theta \in \Theta \subset \IR^{\tilde d}$ with $\int_{\Theta} 1 \dif \theta > 0$ (which excludes non-unique parametrizations). A prominent example being the characteristic function of $\tilde X_0(u)$,
\[
    \phi(u,\theta) := \IE e^{i \theta \tilde X_0(u)},
\]
cf. \cite{leucht2016} ($i$ being the imaginary unit). The cross validation procedure from \reff{def_crossval} can be easily modified to cover such cases by using
\[
    \bar H(G) := \int_{\Theta}H(G_{\theta}) \dif \theta, \quad\quad \hat h^{fun} := \argmin_{h\in H_n}\bar H(\hat G_{\theta,h}^{-}).
\]
Note that here, $\hat G_{\theta,h}^{-}$ naturally depends on $\theta$ through $g_{\theta}$. It is straightforward to show that the following modification of Theorem \ref{thm_crossval} holds if the conditions therein are fulfilled uniformly in $\theta$: Almost surely,
\[
    \lim_{n\to\infty}\frac{d_M^{fun}(h)}{\inf_{h\in H_n}d_M^{fun}(h)}=1,
\]
where $d_M^{fun}(h) := \int_{\Theta} \int_0^{1}|G(u) - \hat G_h(u)|_2^2 w(u) \dif u \dif \theta$.
\end{remark}

\subsection{Simulations}

Since our estimators are model-free, we only discuss the behavior of the method if the underlying time series model is a tvAR(1) process given by
\begin{equation}
    X_{t,n} = a(t/n)\cdot X_{t-1,n} + \zeta_t,\quad t = 1,...,n,\label{model_simulations}
\end{equation}
where $\zeta_t$ are i.i.d. $N(0,1)$ and $a(u) = \sin(2\pi u)$. We have performed similar simulations with tvARCH- and tvMA models leading to similar results. We discuss the following three quantities:
\begin{enumerate}[(a)]
    \item $c(u,1) = \IE \tilde X_0(u) \tilde X_{1}(u)$,
    \item $\gamma(u) = \frac{c(u,1)}{c(u,0)}$ (cf. Remark \ref{remark_correlation}),
    \item $\phi(u,\theta) = \IE \cos(\theta u)$ (cf. Remark \ref{remark_characteristicfunction}, we restrict to the real part for simplicity)
\end{enumerate}
by using the introduced cross validation methods. In all simulations, we use $w(\cdot) = \Ii_{[0.05,0.95]}(\cdot)$ and the Epanechnikov kernel $K(x) = \frac{3}{2}(1-(2x)^2)\Ii_{[-\frac{1}{2},\frac{1}{2}]}(x)$. We use a time series length of $n = 500$ of model \reff{model_simulations} and the set of bandwidths $H_n = \{\frac{k}{50}, k=1,...,50\}$.

\subsubsection{Estimation of $c(u,1)$}
We simulate $N = 2000$ replications and use   $n^{-\alpha} = 0.12$. %\in \{0.08,0.10,0.12\}$.
We choose $\hat h$ to be the largest local minimum of $h \mapsto H(\hat G_h^{-})$.
%which occurs for one the three values of $\alpha$ to avoid instabilities due to nonasymptotic behavior.
In Figure \ref{figure_cv_c1_h} (left) we have plotted a histogram of the bandwidths chosen by our algorithm, together  with the deterministic bandwidth $h_{opt}$ from \reff{hopt_global} which minimizes $\IE d_M(h)$ (and is not available in practice). We observe that $\hat h$ concentrates around $h_{opt}$ with Gaussian shape.
To judge the performance of our procedure, we compare the achieved distances $d_M(h)$ for $h \in \{\hat h, h_{opt}, h^{*}\}$, where $h^{*} := \argmin_{h\in H_n}d_M(h)$ is the bandwidth which minimizes $d_M(h)$ on the current realization. In Figure \ref{figure_cv_c1_h} (right), we have visualized the results with a boxplot. We can see that $\hat h$ performs quite well, the values of $d_M(\hat h)$ are comparable to those of $d_M(h^{*})$.

\begin{figure}
    \centering
    \begin{tabular}{cc}
        \includegraphics[width=5.2cm]{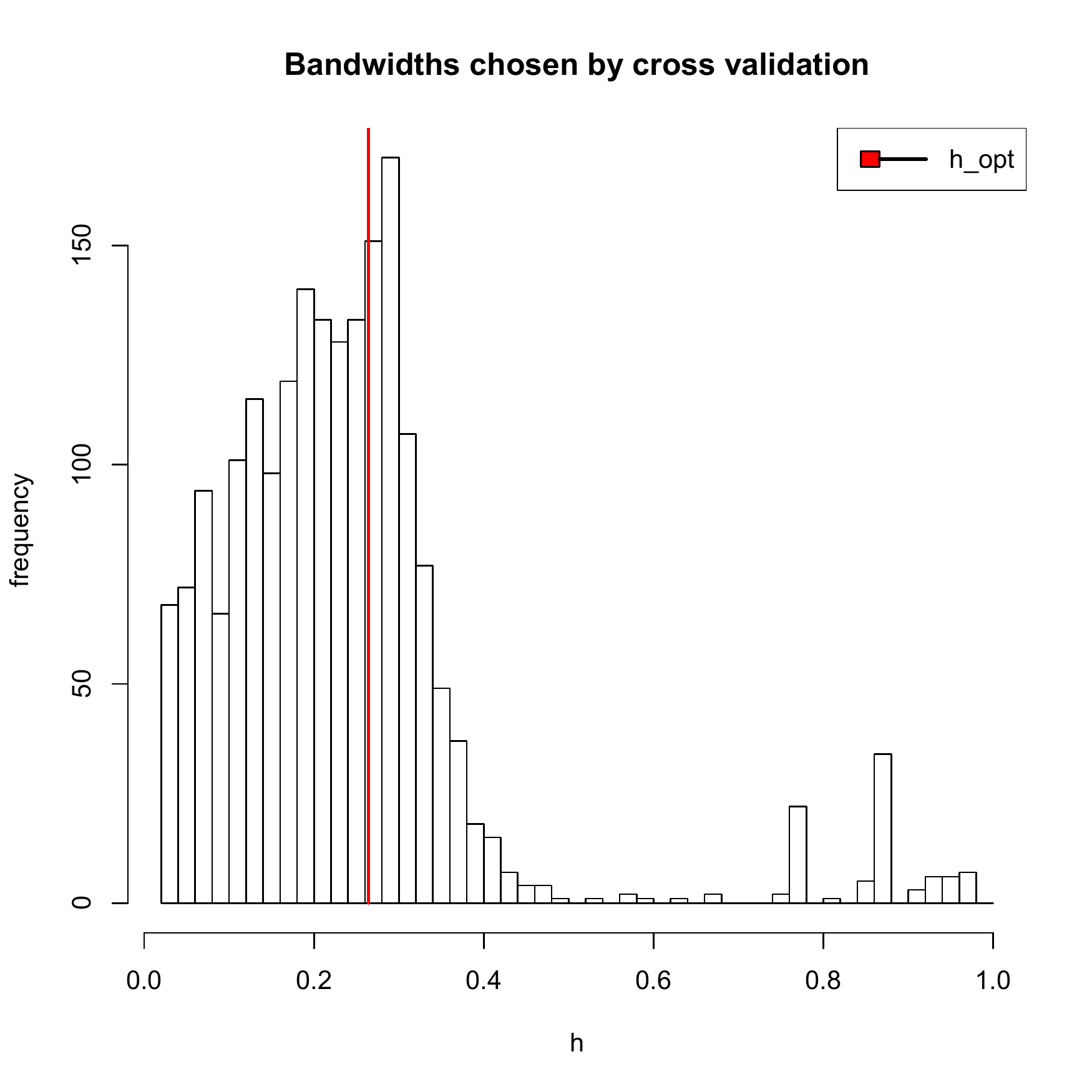} & \includegraphics[width=5.2cm]{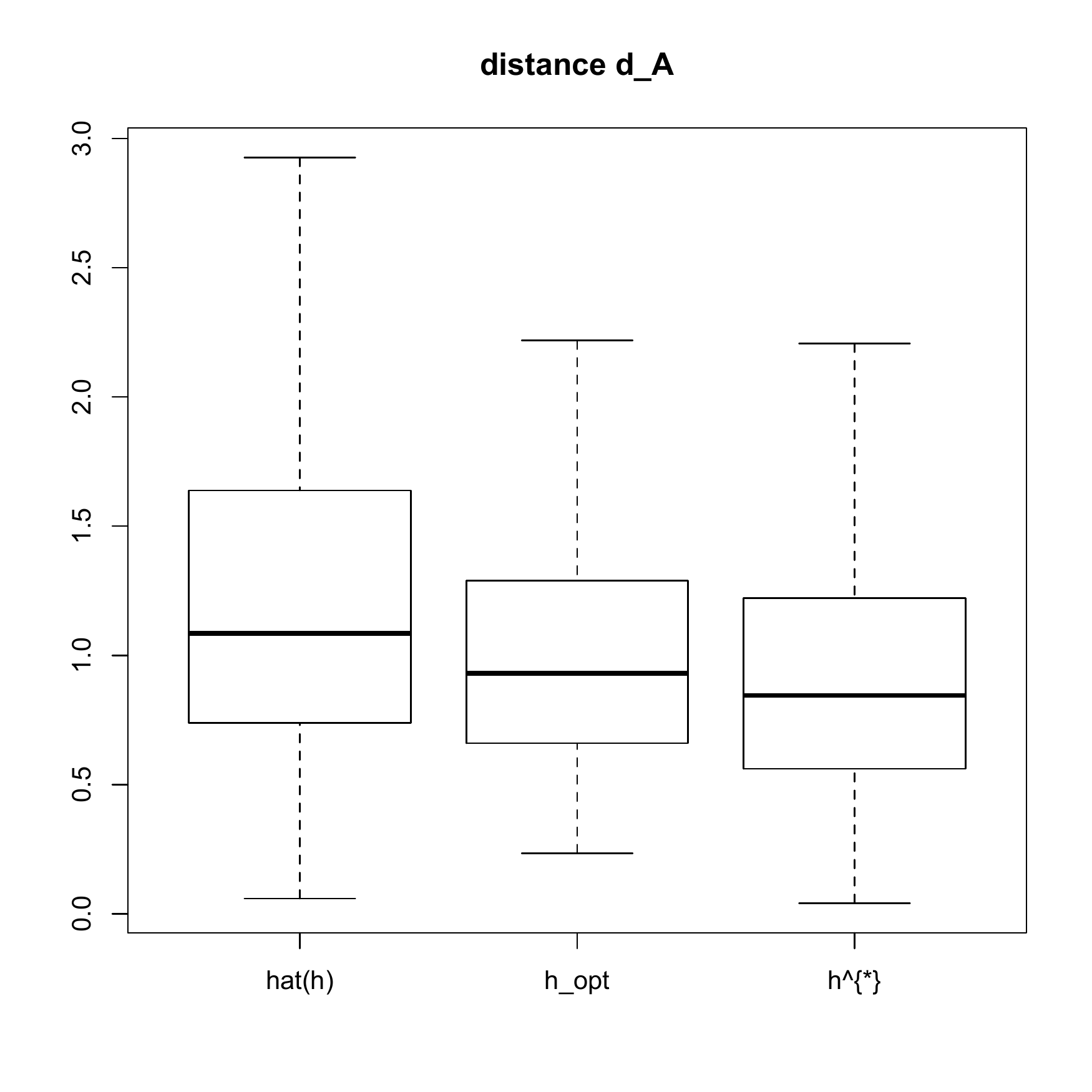}
    \end{tabular}
    \caption{Results for $\hat G_{\hat h}(u)$ used to estimate $c(u,1)$ in the tvAR(1) model \reff{model_simulations}. Left: Histogram of the bandwidths $\hat h$ obtained by cross validation. Right: Boxplot of the distances $d_M(h)$ obtained for $h \in \{\hat h, h_{opt}, h^{*}\}$.}
    \label{figure_cv_c1_h}
\end{figure}

More satisfying results could be obtained by choosing $\alpha$ dependent on the realization. For a fixed realization, it may occur due to nonasymptotic behavior that a single fixed $\alpha$ does lead to cross validation functional $H(\hat G_h^{-})$ (cf. \reff{def_crossval}) which has no local minima. In practice it is therefore recommended to apply the procedure to different $\alpha$ covering the whole interval $(0,1)$ and search for a suitable minimizer to obtain more stable bandwidth choices. In Figure \ref{figure_cv_c1_diffalpha} we have shown the curves $h \mapsto H(\hat G_h^{-})$ obtained for different $\alpha$ such that $n^{-\alpha}$ is ranging from $0.01$ to $0.35$ in steps of 0.01. From the definition of $H(G)$ or $\tilde{H}(G)$ it is obvious that one should look for a local minimum of $H(\hat G_h^{-})$ with respect to $h$ and not for a global minimum (the latter would result in very small values of $h$ such as $h=0.02$ and therefore to an overfitting of the observations $g(Y_{t,n})$). It can be seen in Figure \ref{figure_cv_c1_diffalpha} that for some $\alpha$ no local minimum is available. The whole collection of functions, however, gives a well-founded indication which $h$ to choose, namely an $h$ close to 
$0.08$ which corresponds to the smallest local minimum of $H(\hat G_h^{-})$. Smaller $\alpha$ with a reasonable shape of $H(\hat G_h^{-})$ are preferred since then more observations in the direct neighbourhood of $u$ are used to estimate $\hat G_h^{-}(u)$. This has lead to the choice $n^{-\alpha} = 0.3$ and the resulting local minimum $\hat{h}=0.08$ (from a practical point of view, one may also consider to choose the second local minimum $h \approx 0.3$ where the situation seems to be more stable. This leads to similar results as above).

%such that $n^{-\alpha}$ is ranging from $0.01$ to $0.35$ in steps of 0.01. It can be seen that even if for a specific $\alpha$ no local minimum is available, the whole collection of functions gives a well-founded indication which $h$ to choose. Smaller $\alpha$ with a reasonable shape of $H(\hat G_h^{-})$ are preferred since then the more observations in the direct neighbourhood of $u$ are used to estimate $\hat G_h^{-}(u)$. As it can be seen in Figure \ref{figure_cv_c1_diffalpha}, $h\mapsto H(\hat G_h^{-})$ decreases with smaller $h \in H_n$.  Theorem \ref{thm_crossval} asks for a lower bound placed on $H_n$ which in practice corresponds to search for local minima instead of global minima of $h \mapsto H(\hat G_h^{-})$.

\begin{figure}
    \centering
    \includegraphics[width=7cm]{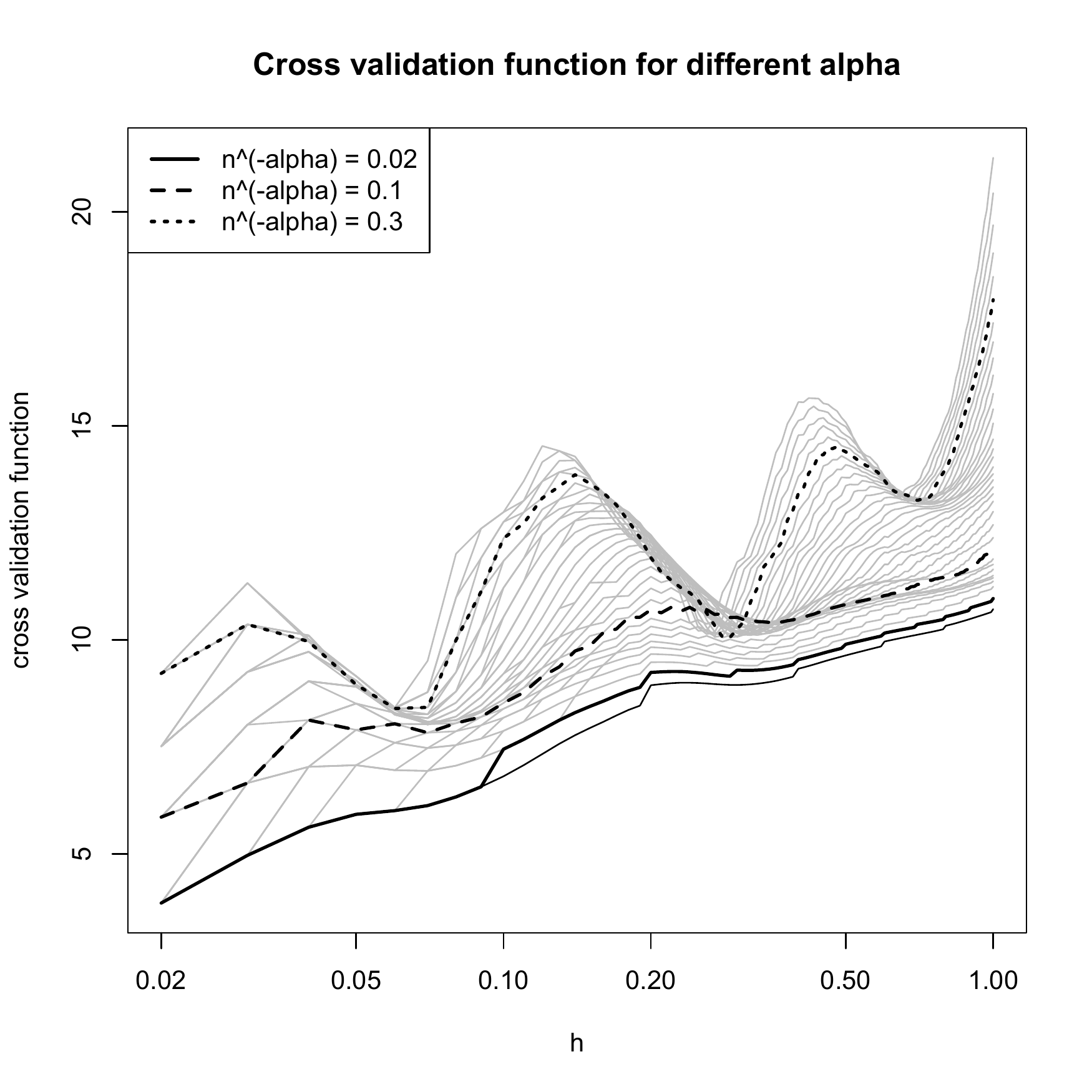}
    \caption{Cross validation functions $H(\hat G_h^{-})$ (grey) for different $\alpha$ where $n^{-\alpha}$ ranges from 0.01 to 0.35 in steps of 0.01.}
    \label{figure_cv_c1_diffalpha}
\end{figure}

\subsubsection{Estimation of $\gamma(u)$}
Writing $G(u) = (c(u,0),c(u,1))$ and $F(x,y) = \frac{y}{x}$, we have $\gamma(u) = \frac{c(u,1)}{c(u,0)} = F(G(u))$ and
\[
    \tilde H(G) = \frac{1}{n}\sum_{t=1}^{n}\frac{X_{t-1,n}^2}{c(t/n,0)^2}\Big(X_{t,n} - \frac{c(t/n,1)}{c(t/n,0)}X_{t-1,n}\Big)^2 w(t/n).
\]
We perform $N = 2000$ replications with $n^{-\alpha} = 0.08$. As before, we compare $\hat h^{comp}$ chosen by cross validation with $h_{opt}^{comp}$ which minimizes the integrated mean-quared error $\IE d_M^{comp}(h)$ and $(h^{*})^{comp} = \argmin_{h\in H_n}d_M^{comp}(h)$.
In Figure \ref{figure_cv_corr_h} (left) we plotted the bandwidths $\hat h^{comp}$ and $h_{opt}^{comp}$ as a vertical red line, in Figure \ref{figure_cv_corr_h} (right) one can see the boxplots corresponding to $d_M^{comp}(h)$ for $h\in \{\hat h^{comp}, h_{opt}^{comp}, (h^{*})^{comp}\}$. As before, $\hat h^{comp}$ concentrates around $h_{opt}^{comp}$ and the distances $d_M^{comp}(\hat h^{comp})$, $d_M^{comp}((h^{*})^{comp})$ are of comparable size.

\begin{figure}
    \centering
    \begin{tabular}{cc}
        \includegraphics[width=5.2cm]{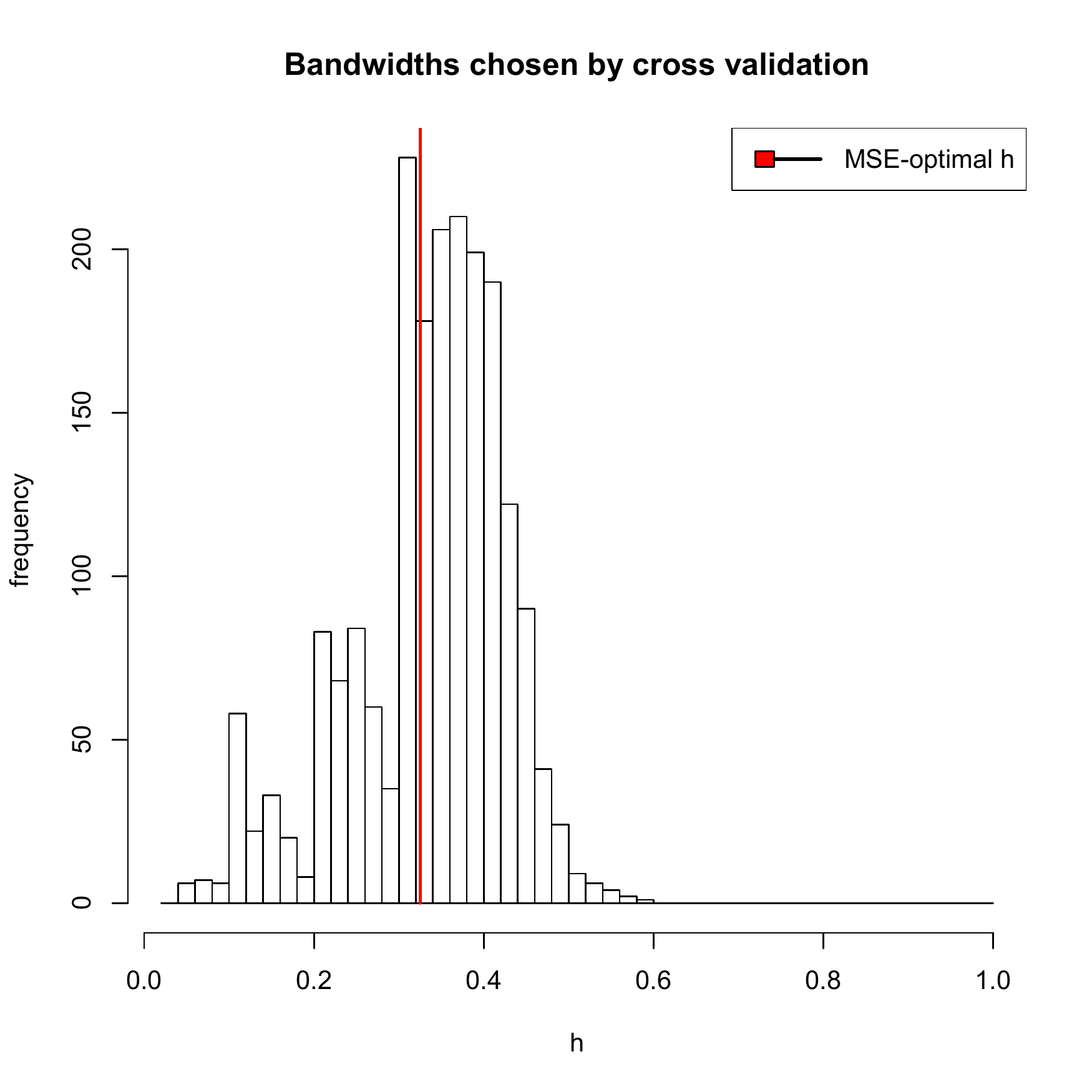} & \includegraphics[width=5.2cm]{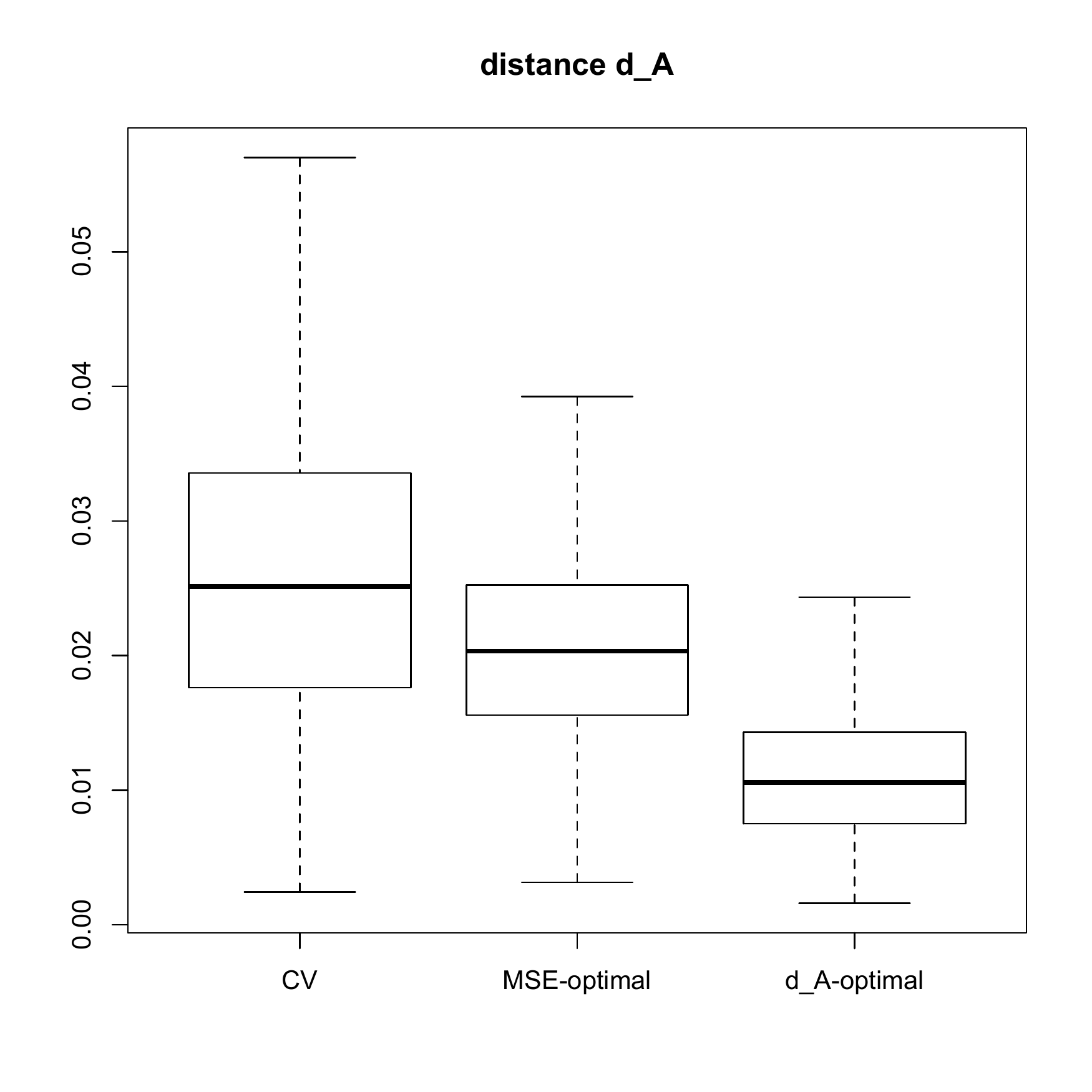}
    \end{tabular}
    \caption{Results for $F(\hat G_{\hat h}(u))$ used to estimate $\gamma(u) = \frac{c(u,1)}{c(u,0)}$ in the tvAR(1) model \reff{model_simulations}. Left: Histogram of the bandwidths $\hat h^{comp}$ obtained by cross validation. Right: Boxplot of the distances $d_M^{comp}(h)$ obtained for $h \in \{\hat h^{comp}, h_{opt}^{comp}, (h^{*})^{comp}\}$.}
    \label{figure_cv_corr_h}
\end{figure}

\subsubsection{Estimation of $\phi(u,\theta)$}
Writing $G_{\theta}(u) = \IE g_{\theta}(\tilde X_t(u))$ with $g_{\theta}(x) = \cos(\theta x)$, we have
\[
    \bar H(G) = \int_{\Theta} \frac{1}{n}\sum_{t=1}^{n}\big|\cos(\theta X_{t,n}) - \hat G_{\theta,h}(u)\big|_2^2 w(t/n) \dif \theta
\]
To perform the simulations, we choose $\Theta = [-10,10]$. Again, we use $N = 2000$ replications with $n^{-\alpha} = 0.10$. We compare $\hat h^{fun}$ chosen by cross validation with $h_{opt}^{fun}$ which minimizes the integrated mean-quared error $\IE d_M^{fun}(h)$ and $(h^{*})^{fun} = \argmin_{h\in H_n}d_M^{fun}(h)$.
In Figure \ref{figure_cv_charfkt_h} (left) we plotted the bandwidths $\hat h^{fun}$ and $h_{opt}^{fun}$ as a vertical red line, in Figure \ref{figure_cv_charfkt_h} (right) one can see the boxplots corresponding to $d_M^{fun}(h)$ for $h\in \{\hat h^{fun}, h_{opt}^{fun}, (h^{*})^{fun}\}$. Again, the results are satisfying in our opinion.

\begin{figure}
    \centering
    \begin{tabular}{cc}
        \includegraphics[width=5.2cm]{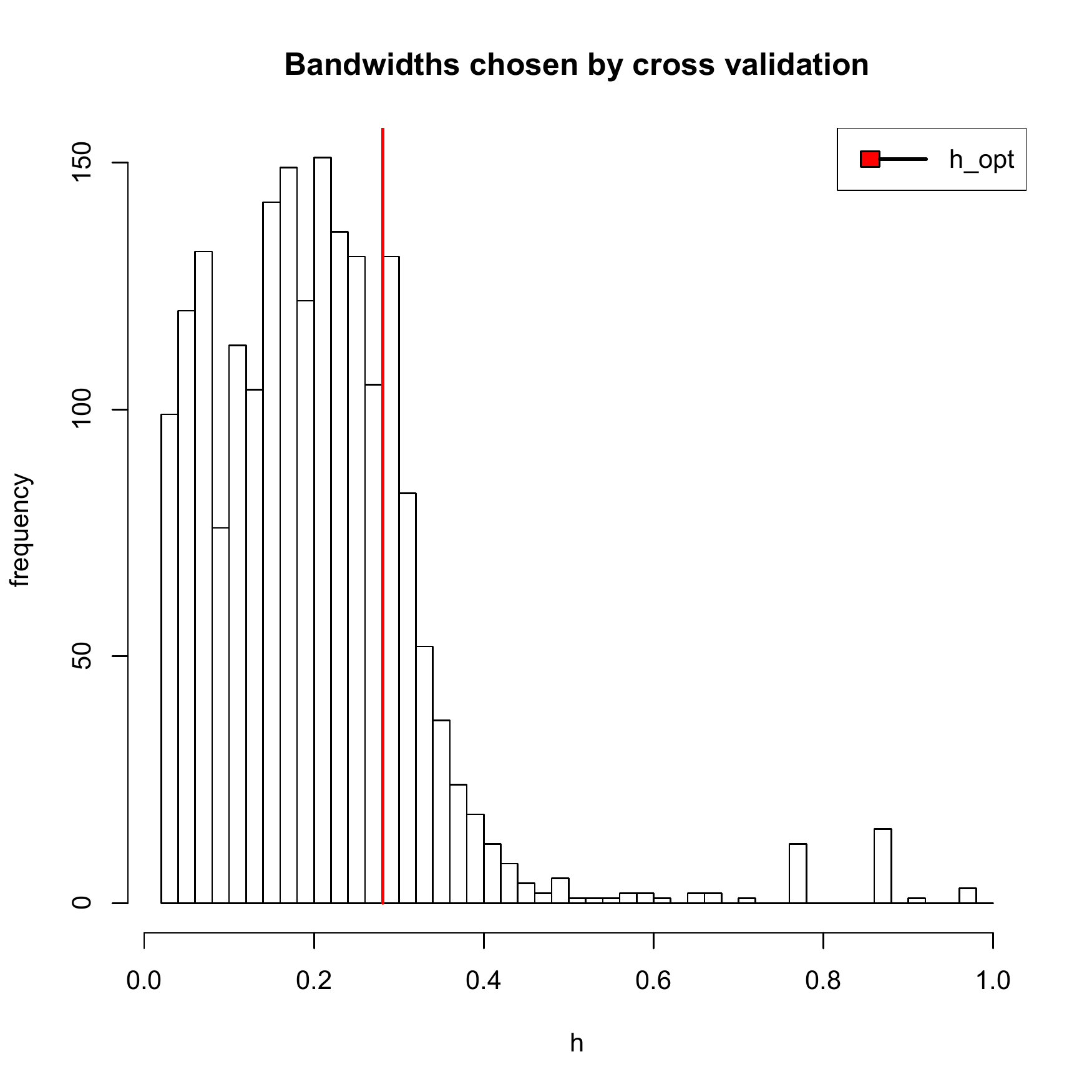} & \includegraphics[width=5.2cm]{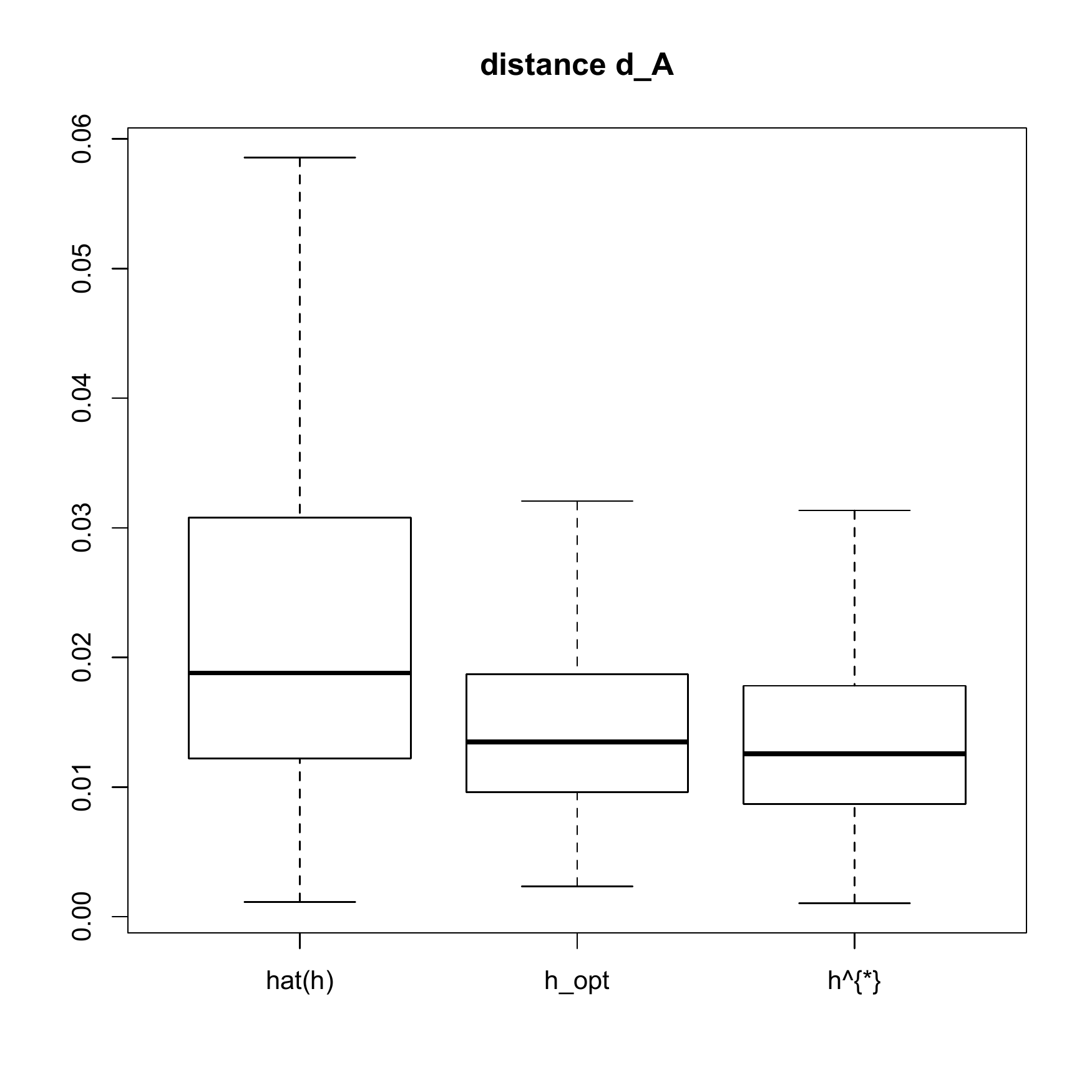}
    \end{tabular}
    \caption{Results for $\hat G_{\theta,\hat h}(u)$ used to estimate $\phi(u,\theta) =\IE \cos(\theta \tilde X_0(u))$ in the tvAR(1) model \reff{model_simulations} uniformly over $\theta \in [-10,10]$. Left: Histogram of the bandwidths $\hat h^{fun}$ obtained by cross validation. Right: Boxplot of the distances $d_M^{fun}(h)$ obtained for $h \in \{\hat h^{fun}, h_{opt}^{fun}, (h^{*})^{fun}\}$.}
    \label{figure_cv_charfkt_h}
\end{figure}

\section{Local model selection: A contrast minimization approach} 
\label{sec4}

The approach presented in the following allows to choose $h$ locally for each $u\in[0,1]$ in the estimator $\hat G_h(u)$. This enables the procedure to take into account local smoothness changes of the function $G(u)$. The algorithm is based on a contrast minimization approach which was introduced by \cite{lepski2011} for a different model. In the following we use a slightly modified estimator for $G(u)$, namely
\[
	\hat G_h^{\circ}(u) := \Big(\frac{1}{n}\sum_{t=1}^{n}K_h(t/n-u)\Big)^{-1}\cdot \hat G_h(u)
\]
which corrects for the deviation of the kernel Riemannian sum from its integral $\int K(x) \dif x = 1$.

The approach is based on a comparison of $\hat G_h(u) - \hat G_{h'}(u)$ with the theoretical variance of $\hat G_h(u)$. This leads to a procedure which does not need to estimate the bias of $\hat G_h(u)$ but only its variance and still finds a approximate minimizer of the corresponding mean squared error. As seen in Theorem \ref{thm_asymptotic_normality}, the asymptotic variance of $\hat G_h(u)$ is connected to the so-called long-run variance of $g(\tilde Y_t(u))$,
\[
	\Sigma(u) = \sum_{k\in\IZ}\Cov(g(\tilde Y_0(u)), g(\tilde Y_k(u)).
\]
Let $\hat \Sigma_n(u)$ be an estimator of $\Sigma(u)$. In practice, the requirements on the quality of $\hat \Sigma_n(u)$ are not too strong, however we need that it has the same order of magnitude as $\Sigma(u)$. Large deviations from the true $\Sigma(u)$ may lead to instable bandwidth choices. For theoretical derivations, we suppose:

\begin{assumption}[Assumptions on $\hat \Sigma_n(u)$]\label{assumption_sigma}
There exists some constant $c_{\Sigma} > 0$ such that for each $u\in[0,1]$,
\begin{itemize}
	\item $\hat \Sigma_n(u)$ is consistent such that for every $\varepsilon > 0$, $\IP(|\hat \Sigma_n(u) - \Sigma(u)|_{\infty} > \varepsilon) = O(n^{-2})$.
	\item $\|\ |\hat \Sigma_n(u)|_{\infty}\ \|_2 \le c_{\Sigma}$.
\end{itemize}
\end{assumption}
\noindent
A possible choice for $\hat \Sigma_n(u)$ is given by
\begin{equation}
	\hat \Sigma_n(u) := \sum_{k=-r_n}^{r_n}\hat c^{g}_\eta(u,k), \quad\quad \hat c_{\eta}^{g}(u,k) := \frac{1}{n}\sum_{t=1}^{n}K_{\eta}(t/n-u)\cdot g(Y_{t,n})g(Y_{t+k,n})'\label{longrun_covariance_estimator}
\end{equation}
with some $0 \le r_n \le n$, $r_n \to \infty$ and $\eta = \eta_n \to 0$. We are now able to define the local bandwidth selection procedure as follows:
	Let $\lambda(h) := \max\{1,\sqrt{\log(1/h)}\}$ and
	\[
		\hat v^2(h,u) := \frac{\sigma_K^2}{nh}\tr(\hat \Sigma_n(u)).
	\]
	Let $H_n \subset (0,1]$ be a geometrically decaying grid of bandwidths given by
	\begin{equation}
		H_n = \{a^{-k}:k\in \IN_0\} \cap [\underline{h}, 1],\label{def_bandwidth_grid}
	\end{equation}
	where $\underline{h} = \underline{h}_n > 0$ is some lower bound specified below.
	For some $C^{\#} > 0$, define
	\begin{equation}
		\hat h(u) :\in \sup\{h\in H_n:|\hat G_h^{\circ}(u) - \hat G_{h'}^{\circ}(u)|_2 \le C^{\#}\cdot \hat v(h',u) \lambda(h') \text{ for all } h' < h, h'\in H_n\}.\label{lokale_wahl_bandbreite}
	\end{equation}

\begin{remark}
	\begin{itemize}
	\item  $\hat v^2(h,u)$ is approximately the variance of $\hat G_h^{\circ}(u) - \hat G_{h'}^{\circ}(u)$. It is multiplied by a log factor $\lambda(h')$ to account for local random deviations. As described in \cite{lepski2011}, the bandwidth \reff{lokale_wahl_bandbreite} can be seen as the largest bandwidth where $\hat G_h^{\circ}(u)$ does not deviate significantly from $G(u)$.
	\item To prove theoretical results, $H_n$ should not allow for too small bandwidths.
	\item The condition $C^{\#} \ge 64$ is needed for the theoretical results. However, for many applications $C^{\#}$ should be chosen smaller to avoid too conservative (i.e. too large) choices of $\hat h(u)$.
	\end{itemize}
\end{remark}

%	For each $u\in[0,1]$, define
%	\begin{eqnarray*}
%		\hat v^2(h,u) &:=& \frac{1}{nh}\int K(x)^2 \dif x \cdot \text{tr}(\hat \Sigma_h(u)),\\
%		\hat v^2(h,h',u) &:=& \frac{1}{n}\int \{K_h(x) - K_{h'}(x)\}^2 \dif x \cdot \text{tr}(\hat \Sigma_h(u)),
%	\end{eqnarray*}
%	where
%	For some fixed $D_1,D_2 > 0$, define
%	\[
%		\lambda(h) := \max\{1,\sqrt{D_2 \log(\overline{h}/h)}\}, \quad\quad \hat \psi(h,h',u) := D_1 \hat v(h,u) \lambda(h) + \hat v(h,h',u)\cdot \lambda(h').
%	\]
%	(where $\overline{h} = \max_{h\in H_n}h$) and choose
%	\[
%		\hat h(u) :\in \sup\{h\in H_n:|\hat G_h(u) - \hat G_{h'}(u)|_2 \le \hat \psi(h,h',u) \text{ for all } h' < h, h'\in H_n\}.
%	\]
%\end{definition}

We now provide the assumptions on the process $\tilde X_0(u)$ under which the theoretical statements for the local bandwidth selection procedure $\hat h(u)$ holds.

\begin{assumption}[Moment / Dependence assumptions]\label{assumption_dependence}
	Let $\alpha \ge 0$. Let Assumption \ref{def_localstat} hold for all $q \ge 1$. Define $N_{\alpha}(q) := \Gamma(\alpha q + 2)$. Assume that for all $u\in [0,1]$,
	\begin{enumerate}
		\item[(i)] $\sup_{u\in[0,1]}\|\tilde X_0(u)\|_q \le D \cdot N_{\alpha}(q)$,
		\item[(ii)] $\sup_{u\in[0,1]}\delta_q^{\tilde X_0(u)}(k) \le D \cdot N_{\alpha}(q) \cdot \rho^{k}$.
	\end{enumerate}
\end{assumption}

\begin{remark}
	\begin{itemize}
		\item Assumption \ref{assumption_dependence}(i) basically asks for a quantification of the growth of the moments $\|\tilde X_0(u)\|_q^q$ in $q$. It can be easily seen that the condition $\sup_{u\in[0,1]}\|\tilde X_0(u)\|_q \le D \cdot N_{\alpha}(q)$ follows if $\tilde X_0(u)$ has a Lebesgue density $\sim \exp(-x^{1/\alpha})$.
		\item Assumption \ref{assumption_dependence}(ii) additionally asks the process to have geometrically decaying dependence measure $\delta_q^{\tilde X_0(u)}(k)$.
	\end{itemize}
\end{remark}

The assumptions stated above are mainly used to prove a Bernstein inequality which is a key ingredient to prove the consistency of contrast minimization methods. The geometric dependence decay is used to apply a Bernstein inequality from \cite{doukhanneumann2007}, while the moment conditions are necessary to allow for a large set of bandwidths $H_n$ by establishing a simple exponential inequality. The Bernstein inequality is formulated in terms of the process
\[
    \tilde G_h(u) := \frac{1}{n}\sum_{t=1}^{n}K_h(t/n-u)\cdot g(\tilde Y_t(t/n))
\]
which is a natural approximation of $\hat G_h(u)$ itself. The deviation $\hat G_h(u) - \tilde G_h(u)$ can be controlled by using the property $\|X_{t,n} - \tilde X_t(t/n)\|_q \le Dn^{-1}$ from Assumption \ref{def_localstat}.

\begin{theorem}[Bernstein inequality for $\tilde G_h(u)$]\label{lemma_bernstein_kerne}
	Fix $u\in [0,1]$. Assume that $g\in \sH(M,\chi,C)$ and Assumption \ref{assumption_dependence} holds. Then there exist some constants $\tilde c_1,c_4,c_5,c_H > 0$ only dependent on $M,\chi,C,D,\rho,|K|_{\infty},\alpha$ such that for all $j = 1,...,d$, $h\in H_n$:
	\[
		\IP\Big((nh)\big|\tilde G_h(u)_j - \IE \tilde G_h(u)_j\big| > \gamma\Big) \le 2\exp\Big(-\frac{\gamma^2}{32 (nh)^2 v_j^2(h,u)+ c_4 a_n^{1/3} \gamma^{5/3}}\Big) + c_5 \frac{n^{-1}}{\gamma},
	\]
	with $a_n := \tilde c_1(8\log(n))^{\alpha M}$, $v_j^2(h,u) := \frac{\sigma_K^2}{nh}\Sigma(u)_{jj}$ and $\underline{h} = c_H\cdot \log(n)^{5+2\alpha M}$ in \reff{def_bandwidth_grid}.
\end{theorem}

The following theorem states that $\IE| \hat G_{\hat h(u)}(u) - G(u)|_2^2$ is bounded by an universal constant times $\IE|\hat G_{h_{opt}(u)}(u) - G(u)|_2^2$ up to an additional log-factor (cf. \reff{hopt_lokal} for $h_{opt}(u)$).

\begin{theorem}\label{upperbound_thm2}
	Fix $u\in (0,1)$. Suppose that Assumptions \ref{assumption_sigma} and \ref{assumption_dependence} are fulfilled. 
	
	Suppose further that $g \in \sH(M,\chi,C)$ such that for all $j \ge 0$, $|\chi_j| \le C \rho^j$ and $G$ is twice continuously differentiable around $u$ with  $|\partial_u^2 G(u)|_2 \not= 0$.
	
	Then there exists some universal constant $c = c(a) > 0$ and some constant $c_H > 0$ only dependent on $M,\chi,C,D,\rho,|K|_{\infty},\alpha$ such that for $\underline{h} = c_H\cdot \log(n)^{5+2\alpha M}$, $C^{\#} \ge 64$ and $n$ large enough,
	\[
		\IE |\hat G_{\hat h(u)}^{\circ}(u) - G(u)|_2^2 \le c \cdot \Big(\sigma_K^2 \tr(\Sigma(u))\Big)^{4/5}\cdot \Big(\mu_K^2 |\partial_u^2 G(u)|_2^2\Big)^{1/5}\cdot \Big(\frac{\log(n)}{n}\Big)^{4/5}.
	\]
\end{theorem}

\begin{remark}
    The parameter $C^{\#}$ is a tuning parameter of the procedure. However, from a theoretical point of view it does not depend on unknown quantities like $G(u)$ but is a universal constant and therefore can be chosen based on training data before applying the procedure to real data. In practice, we obtained good results with $C^{\#} \in [0.5,1.0]$.
\end{remark}

\subsection{Simulations}

We discuss the behavior of $\hat h(u)$ if the underlying time series model is a tvAR(1) process given by \reff{model_simulations} with $\zeta_t$  i.i.d. $N(0,1)$ and
\begin{enumerate}[(a)]
    \item $a(u) = 0.6 \sin(2\pi u)$,
    \item $a(u) = 0.5 \cdot (1 - 2\Ii_{[0.5,1]}(u))$
\end{enumerate}
In both situations we will estimate
\[
    c(u,1) = \IE \tilde X_0(u) \tilde X_{1}(u)
\]
by using $N = 2000$ replications of the time series with length $n = 2000$. In the simulations, we use the Epanechnikov kernel $K(x) = \frac{3}{2}(1-(2x)^2)\Ii_{[-\frac{1}{2},\frac{1}{2}]}(x)$ and the set of bandwidths $H_n$ with $a = 0.9$ and $\underline{h} = 0.9^{29} \approx 0.047$ with $C^{\#} = 1.0$ (for (a)) and $C^{\#} = 0.8$ (for (b)). We estimate the long run variance $\Sigma(u)$ by $\hat \Sigma_n(u)$ given in \reff{longrun_covariance_estimator} with a simple ad-hoc choice of $\eta = 0.35$ and $r_n = 18$.

A typical phenomenon arises if $\underline{h}$ is chosen too small, i.e. $H_n$ allows for too small bandwidths: Then, $\hat h(u)$ tends to select the smallest possible bandwidth abruptly at random locations $u \in [0,1]$. The theoretical reason being that in such cases $\underline{h}$ does not fulfill $\underline{h} \ge c_H \frac{\log(n)^{5+2\alpha M}}{n}$ which is needed to discuss large deviations of $\hat G_h(u)$ in the Bernstein inequality. In practice, one may apply the selection procedure $\hat h(u)$ for all $u\in[0,1]$ and for different $\underline{h}$, choosing the smallest $\underline{h}$ where no abrupt outliers occur.

In Figure \ref{figure_lepski_onerealization} (left) we depicted the typical behavior of $\hat h(u)$ (red) for the two scenarios (a),(b) for a single realization. Moreover, $\hat h(u)$ itself is drawn with a green line. The rough shape of $G_{\hat h(u)}(u)$ is due to the  stepwise form of $\hat h(u) \in H_n$.
Generally speaking it can be seen that $\hat h(u)$ gets smaller if a too large bandwidth would introduce an unnecessary bias (see the peaks in (a) or the step in (b)). $\hat h(u)$ gets large if the variance of the procedure can be significantly reduced. In (b), this is the case on the boundaries where a nearly constant function has to be estimated and in the exact middle at $u = 0.5$ where the target value 0 is estimated best by taking the mean over all observations. In (a) this is the case at points $u$ where $G(u)$ has nearly linear shape (turning points). A typical drawback of the local selection procedure $\hat h(u)$ is that it may be more sensitive to strong local dependencies. In (b) one can see that around $u \approx 0.8$, $\hat h(u)$ is much smaller than the bandwidths chosen around it which is due to a strong peak of the process which is due to a violation of the asymptotic statements. 
In Figure \ref{figure_lepski_onerealization} (right) the corresponding covariance estimator $\hat \Sigma_n(u)$ is drawn against the true long run variance $\Sigma(u)$. It can be seen that the method works quite satisfying even if $\Sigma(u)$ is not estimated well; however it is important that the order of magnitudes coincide.
Comparing $G_{\hat h(u)}(u)$ (red) and $G_{h^{*}}(u)$ (blue), where $h^{*} = \argmin_{h\in H_n}d_M(h)$, it can be seen that in case (b), the local bandwidth selection may outperform the global bandwidth selection because of the significantly changing smoothness properties of $G(u)$.

\begin{figure}
    \centering
    \begin{tabular}{ccc}
        (a) &  \includegraphics[width=5.2cm]{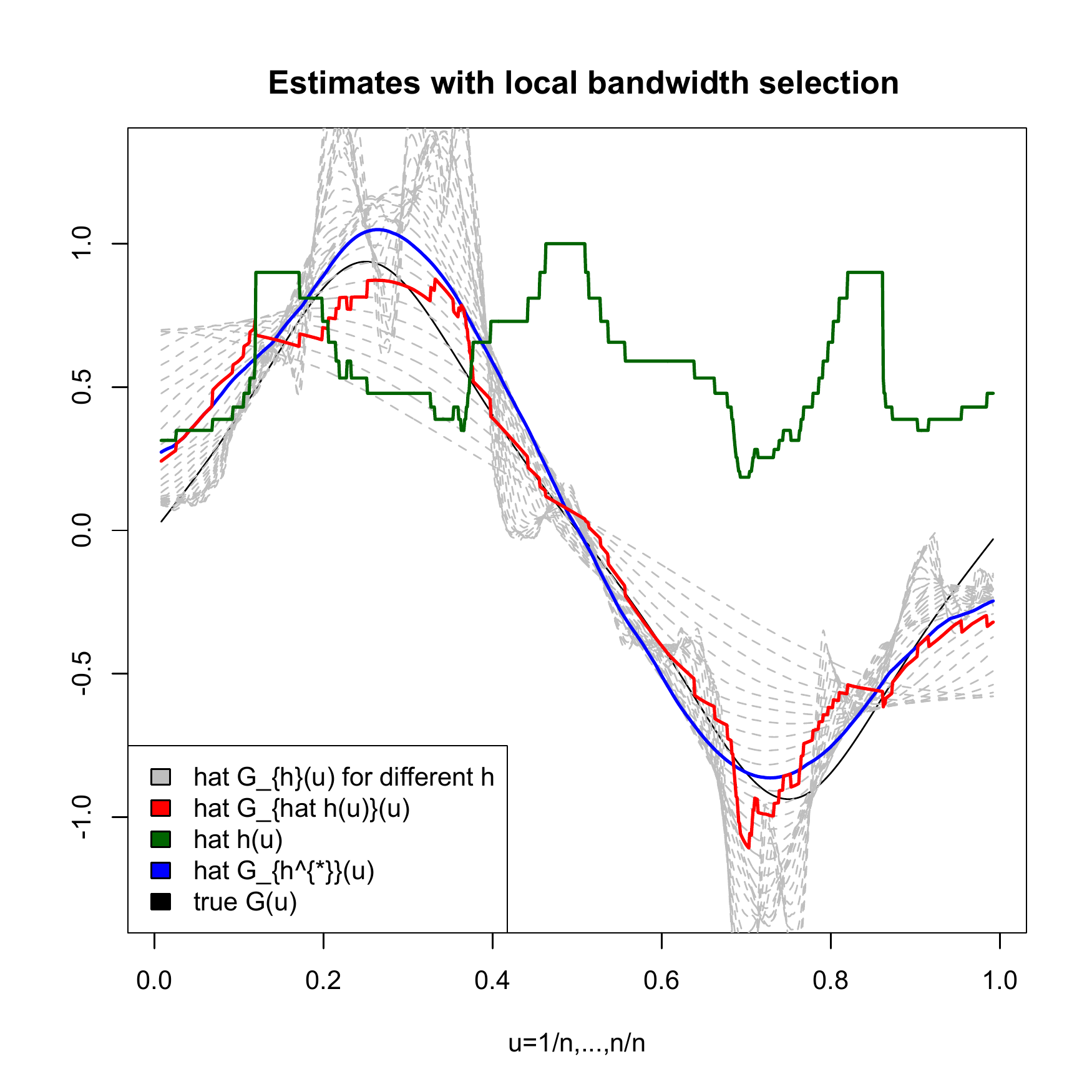}& \includegraphics[width=5.2cm]{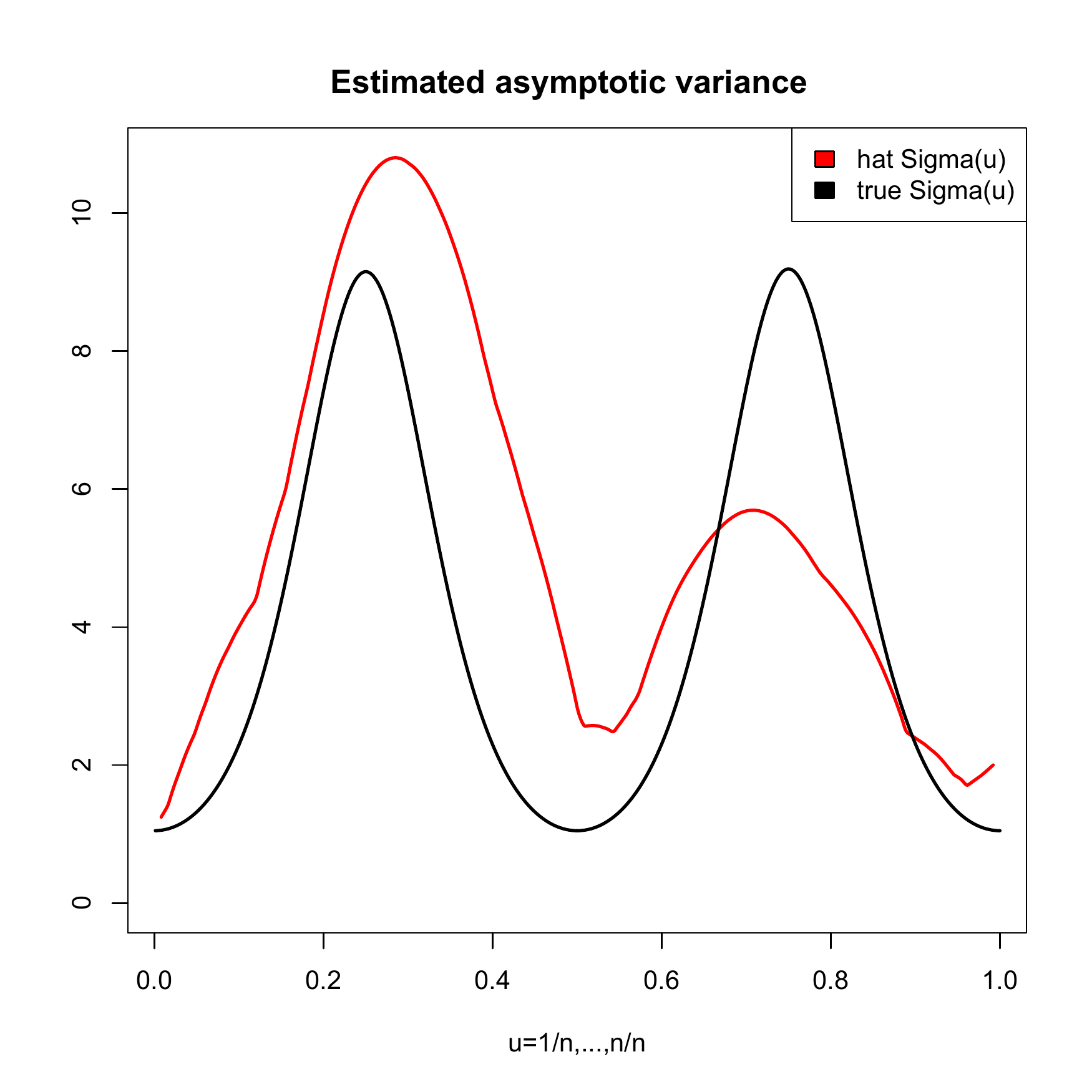}\\
        (b) &  \includegraphics[width=5.2cm]{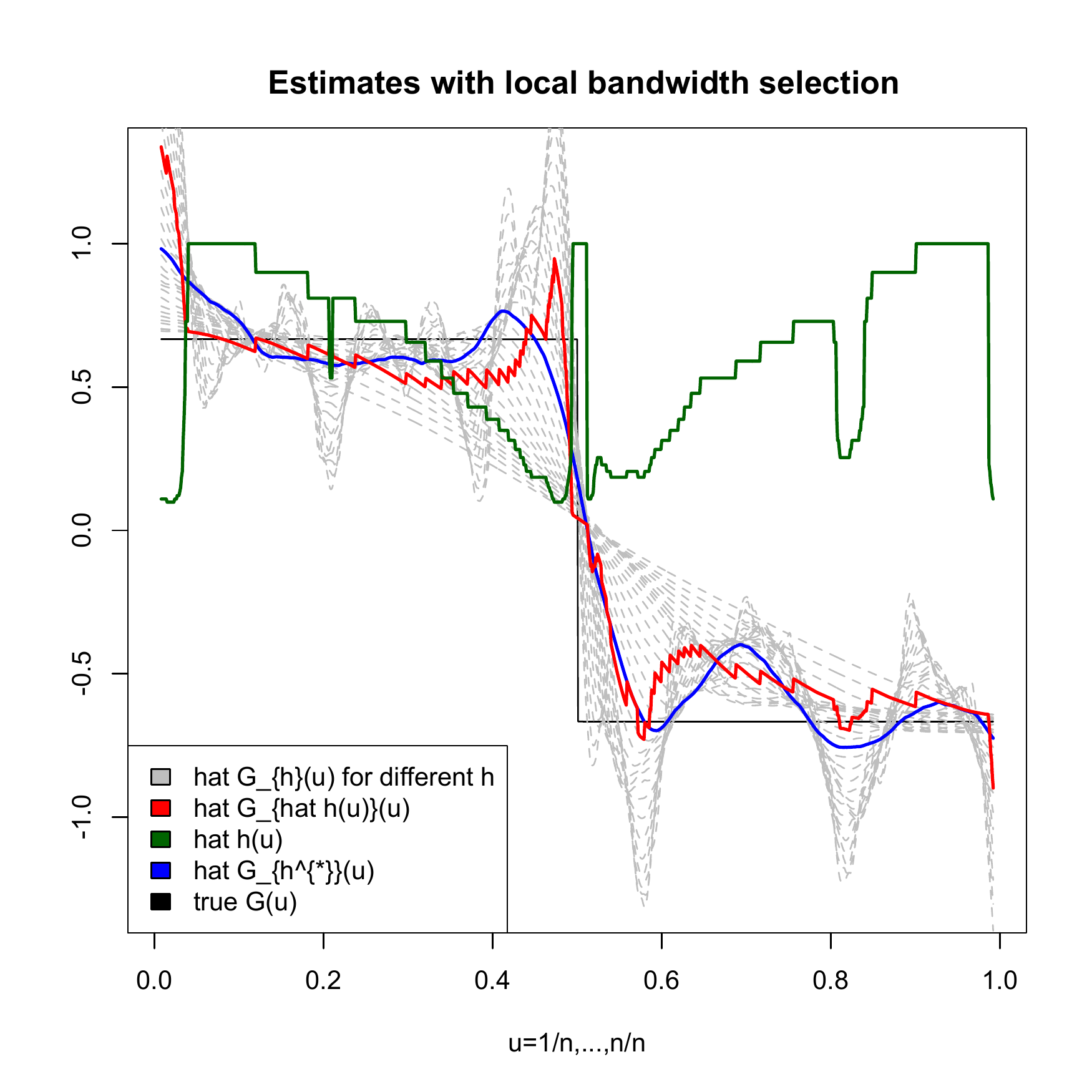}& \includegraphics[width=5.2cm]{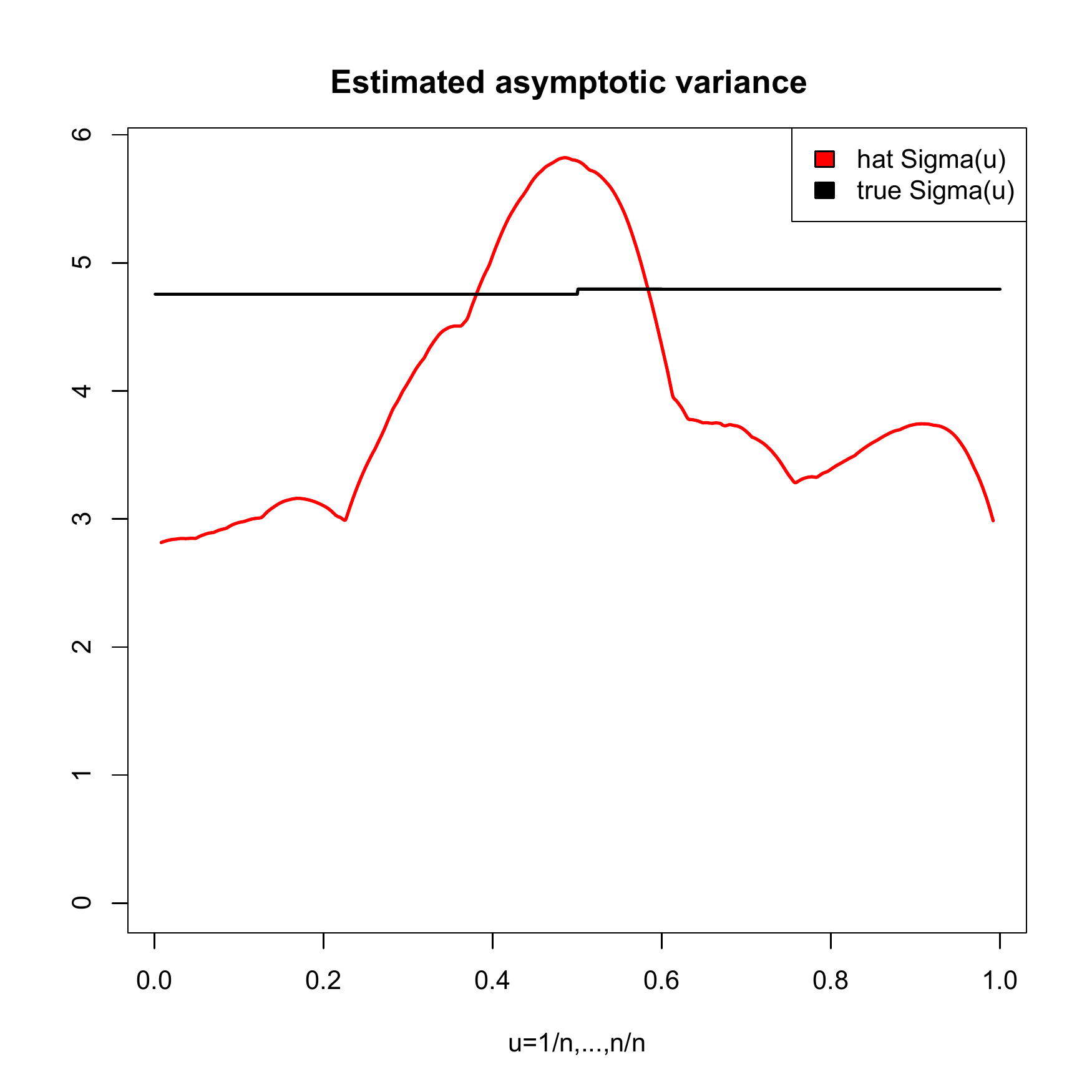}
    \end{tabular}
    \caption{Scenario (a),(b): Behavior of the algorithm for one realization. Left: Plot of $G_{\hat h(u)}(u)$ (red) together with all $G_h(u)$, $h\in H_n$ (grey) and the global optimal choice $G_{h^{*}}(u)$ (blue) against the true value $G(u)$ (black). The green line is $\hat h(u)$. Right: True $\Sigma(u)$ (black) and its estimator $\hat \Sigma_n(u)$ (red).}
    \label{figure_lepski_onerealization}
\end{figure}

In the following we choose the weight function $w(u) = \Ii_{[0.05,0.95]}(u)$ and compare the integrated distance $d_M^{local}(h(u)) := \int_0^{1} |\hat G_{h(u)}(u) - G(u)|_2^2 w(u) \dif u$ for the local selector $h(u) = \hat h(u)$ and the global optimal selector $h(u) = h^{*}$ which is not known in practice. The results are summarized in Figure \ref{figure_lepski_all} via boxplots (right) and 0.05- and 0.95-quantile curves (left). It should be noted that even comparable results of $d_M^{local}(\hat h(u))$ with $d_M^{local}(h^{*})$ are worth to be emphasized since $\hat h(u)$ suffers from the typical higher variance of a data-driven selection procedure which is not the case for $h^{*}$.
In case (a), we see that the quantile curves of both selection procedures are comparable with a little bit higher variance around the peaks . The global optimal bandwidth clearly outperforms $\hat h(u)$ since $G$ has nearly smoothness properties over the whole time line.

In the case (b), it can be seen from the quantile curves that for $u$ near $0$ and $1$, the local selector outperforms the global optimal selector since it is able to choose larger bandwidths there, reducing the variance of $\hat G_{\hat h(u)}(u)$. For $u$ around $0.5$, $\hat G_{\hat h(u)}(u)$ manages to mimic the step of the true function $G(u)$ better than the global selector $h_{opt}(u)$.

\begin{figure}
    \centering
    \begin{tabular}{ccc}
        (a) &  \includegraphics[width=5.2cm]{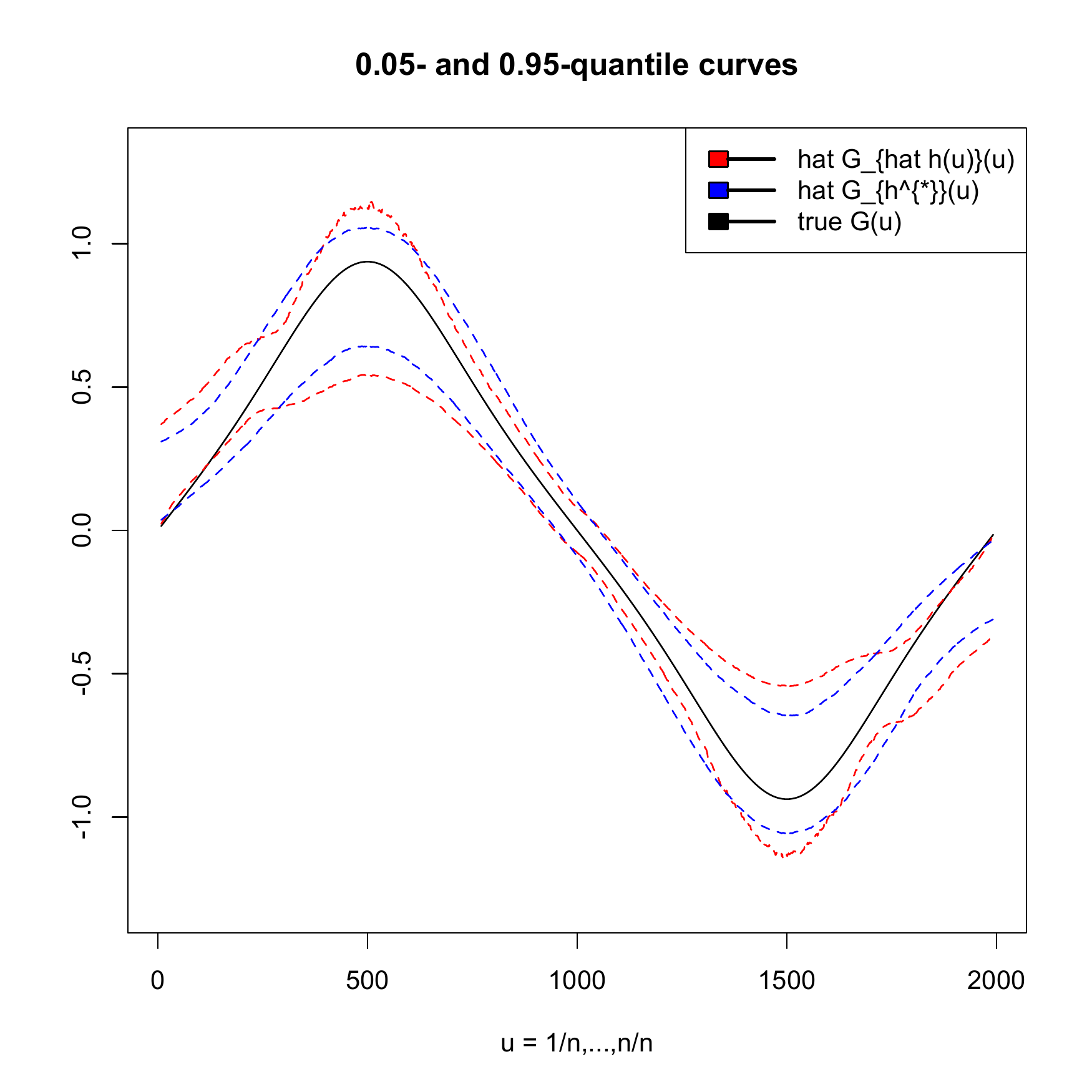}& \includegraphics[width=5.2cm]{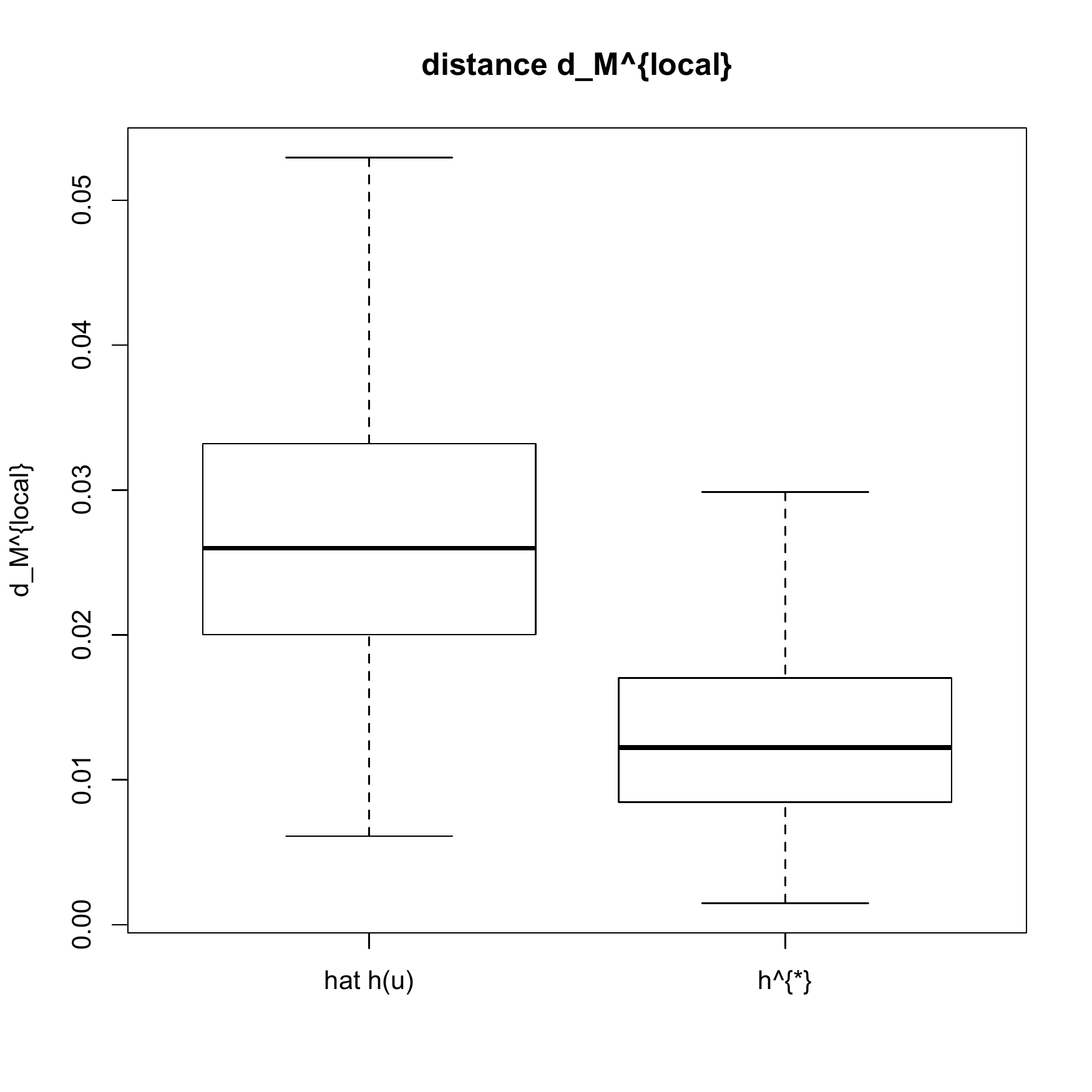}\\
        (b) &  \includegraphics[width=5.2cm]{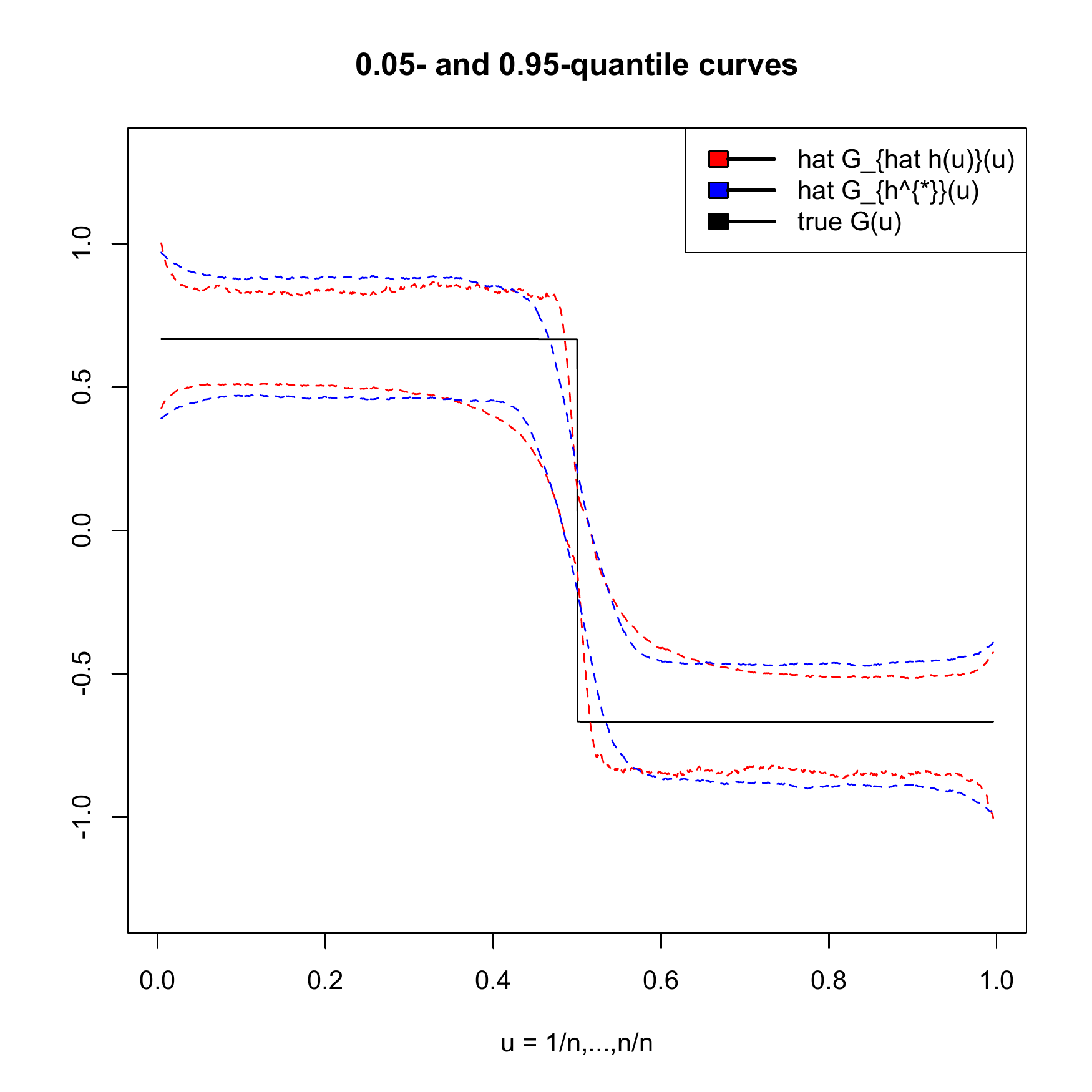}& \includegraphics[width=5.2cm]{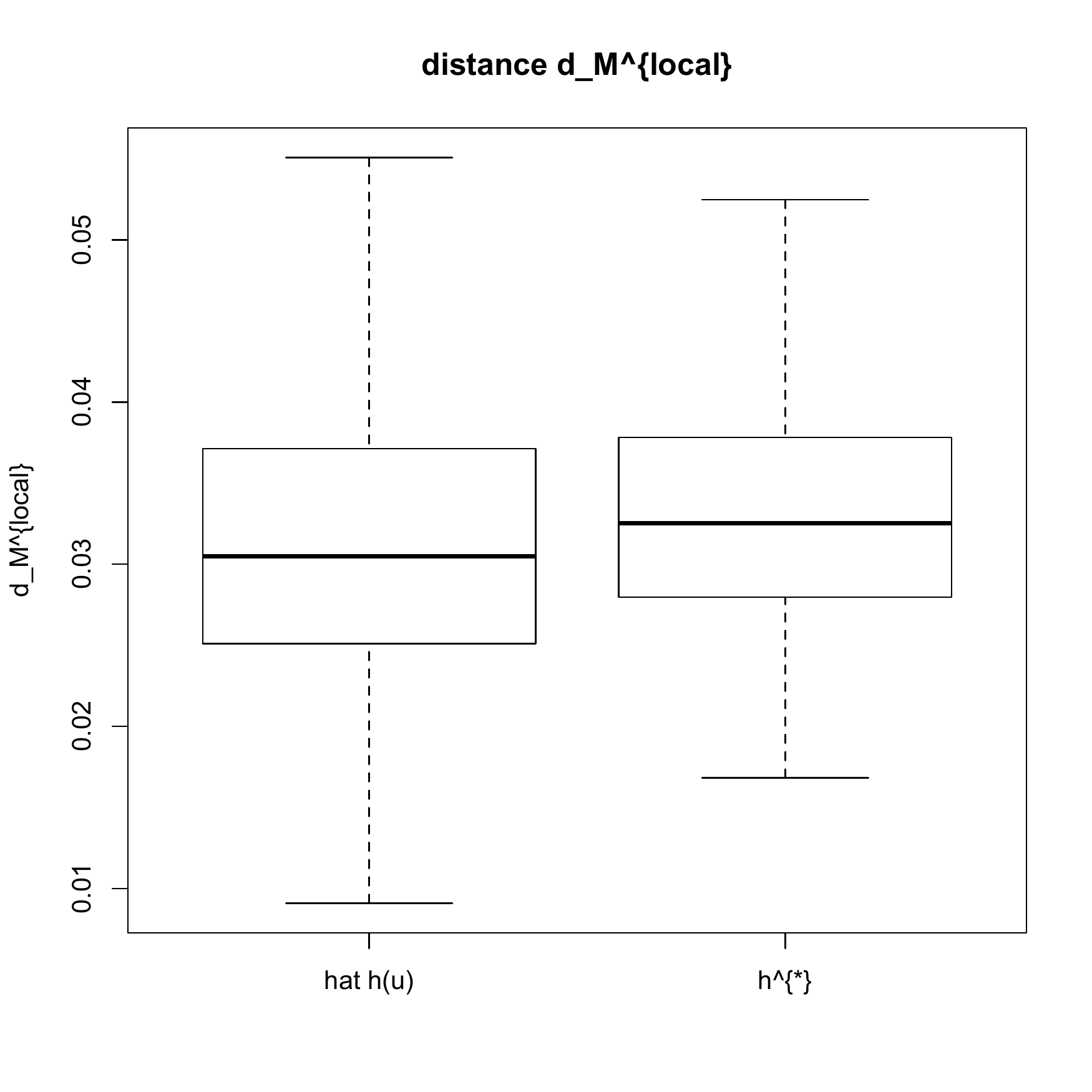}
    \end{tabular}
    \caption{Scenario (a),(b): Behavior of the local bandwidth selector for $N = 2000$ replications. Left: 0.05- and 0.95-quantile curves of $G_{\hat h(u)}(u)$ (red) and the global optimal choice $G_{h^{*}}(u)$ (blue) against the true value $G(u)$ (black). Right: Boxplots of the achieved distances $d_M^{local}(\hat h(u))$ and $d_M^{local}(h^{*})$.}
    \label{figure_lepski_all}
\end{figure}

\section{Conclusion}
\label{sec5}
In this paper, we have developed two methods for bandwidth selection for nonparametric moment estimators of locally stationary processes of some curve $G(u)$. We have derived theoretical results for their optimality with respect to mean-squared-error type distance measures and found with simulations that they work quite well in practice. The first method is  based on a cross validation approach and allows for global bandwidth selection. A critical issue is to deal with the dependency of the observed time series which is controlled by some  quantity $\alpha$. We have given theoretical and heuristical motivations how to choose this value to obtain a stable procedure. It should be emphasized that this choice is rather  straightforward and data-dependent and therefore should not be considered as a tuning parameter. In a series of remarks, we have shown how to generalize our method to more general settings, for instance if $G(u)$ is a composition of moment functions.

The second method is for local bandwidth selection, i.e. for each time point a different bandwidth is chosen. This allows for taking into account local smoothness properties of the unknown curve. The method needs an estimator $\hat \Sigma_n(u)$ of the asymptotic long-run variance and is dependent on a parameter $C^{\#}$. We have seen that the quality of $\hat \Sigma_n(u)$ does not influence the bandwidth selection procedure very much, while the choice of $C^{\#}$ however is crucial to obtain meaningful results. In our theoretical results we have shown that $C^{\#}$ is some universal constant which does not need to depend on the unknown quantities to be estimated. However, the conditions stated in theory may lead to too conservative estimates (i.e. the variance is reduced too strongly by introducing a large bias). In our simulations we have given ad hoc choices of $C^{\#}$ which work fairly well, but in more specific applications it may be necessary to adjust $C^{\#}$ further. We have seen in simulations that the local bandwidth selection procedure may outperform global bandwidth selection procedures. Finally, it should be noted that the presented local bandwidth selection procedure reduces the problem of choosing several  tuning parameters separately (all the local bandwidths $h(u)$, $u\in[0,1]$) to a proper choice of $C^{\#}$.

In practice, we propose to first use the cross validation method for a first guess $\hat h$ of the global optimal bandwidth. Afterwards one may  apply the local selection procedure, using $\hat h$ (or the estimator of $G(u)$ based on $\hat h$, respectively) to calibrate $C^{\#}$.

One may try to improve the theory for the presented methods, allowing for more general structures of $G(u)$ and its estimators or investigating $G(u)$ with moments of two sided functions. These questions are left to further research.

\section{Appendix}
\label{sec6}

During the proofs, $c' > 0$ is a generic constant which may depend on
\begin{itemize}
    \item $M,\chi,C$ (from $g \in \sH(M,\chi,C)$), 
    \item $D,\alpha,\rho$ (from Assumption \ref{def_localstat} and Assumption \ref{assumption_dependence}),
    \item $K \in \sK$ (the kernel).
\end{itemize}
Its value may change from line to line. We define
\[
    \tilde G_h(u) := \frac{1}{n}\sum_{t=1}^{n}K_h(t/n-u)\cdot g(\tilde Y_t(t/n))
\]
and $\tilde G_h^{\circ}(u) := (\frac{1}{n}\sum_{t=1}^{n}K_h(t/n-u))^{-1}\tilde G_h(u)$.

\begin{proof}[Proof of Theorem \ref{theorem_bias}] (i) By Lemma \ref{lemma_statapprox}(ii), we have in each component that
%We have in each component that
%\[
%	\Big\|\hat G_h(u) - \frac{1}{n}\sum_{t=1}^{n}K_h(t/n-u)\cdot g(\tilde Y_t(t/n))\Big\|_1 \le  \frac{1}{n}\sum_{t=1}^{n}|K_h(t/n-u)|\cdot \|g(Y_{t,n}) - g(\tilde Y_t(t/n))\|_1.
%\]
%By Hoelder's inequality,
%\begin{eqnarray*}
%	&& \|g(Y_{t,n}) - g(\tilde Y_t(t/n))\|_1\\
%	&\le&  C (1 + 2C^{M-1})\cdot \Big[\sum_{i=1}^{t}\chi_i \|X_{t-i+1,n} - \tilde X_{t-i+1}((t-i+1)/n))\|_M\\
%	&&\quad\quad\quad + \sum_{i=1}^{t}\chi_i \|\tilde X_{t-i+1}((t-i+1)/n)) - \tilde X_{t-i+1}(t/n)\|_M + \sum_{i=t+1}^{\infty}\chi_i \|\tilde X_{t-i+1}(t/n)\|_M\Big]\\
%	&\le& C^2(1+2C^{M-1})\cdot \Big[n^{-1} \cdot \sum_{i\in\N}\chi_i + n^{-1}\sum_{i\in\N}i\cdot \chi_i + \sum_{i=t+1}^{\infty}\chi_i\Big],
%\end{eqnarray*}
%such that
\[
	\sup_{u\in[0,1]}\Big\|\hat G_h(u) - \frac{1}{n}\sum_{t=1}^{n}K_h(t/n-u)\cdot g(\tilde Y_t(t/n))\Big\|_1 \le c' (nh)^{-1}.
\]
Since $u \mapsto \tilde X_t(u)$ is twice differentiable,  the same holds for each component of $u \mapsto g(\tilde Y_t(u))$ by using the chain rule. We conclude that
\begin{eqnarray*}
	g(\tilde Y_t(u')) &=& g(\tilde Y_t(u)) + (u' - u)\cdot \partial_u g(\tilde Y_t(u)) + \frac{1}{2}(u'-u)^2 \cdot \partial_u^2 g(\tilde Y_t(u))\\
	&&\quad + \int_{u}^{u'}(s-u) \cdot \big\{\partial_u^2 g(\tilde Y_t(s)) - \partial_u^2 g(\tilde Y_t(u))\big\} \dif s,
\end{eqnarray*}
and
\begin{eqnarray}
	&&\IE\Big[\frac{1}{n}\sum_{t=1}^{n}K_h(t/n-u)\cdot g(\tilde Y_t(t/n))\Big] - G(u)\nonumber\\
	&=& \Big(\frac{1}{n}\sum_{t=1}^{n}K_h(t/n-u)\cdot \IE g(\tilde Y_t(u)) - G(u)\Big)\\
	&&\quad\quad + \frac{1}{n}\sum_{t=1}^{n}K_h(t/n-u)\cdot (t/n-u)\cdot \IE \partial_u g(\tilde Y_t(u))\nonumber\\
	&&\quad\quad + \frac{1}{2n}\sum_{t=1}^{n}K_h(t/n-u)\cdot (t/n-u)^2\cdot \IE \partial_u^2 g(\tilde Y_t(u))\nonumber\\
	&&\quad\quad + \frac{1}{n}\sum_{t=1}^{n}K_h(t/n-u)\int_{u}^{t/n}(s-u)\IE\big[\partial_u^2 g(\tilde Y_t(s)) - \partial_u^2 g(\tilde Y_t(u))\big] \dif s.\label{theorem_bias_eq1}
\end{eqnarray}
Due to Lipschitz continuity and symmetry of $K$, the first two terms in \reff{theorem_bias_eq1} are $\le \text{const.} (nh)^{-1}$. The third term in \reff{theorem_bias_eq1} is $\frac{h^2}{2}\mu_K \cdot \partial_u^2 G(u) + O((nh)^{-1})$.

Due to Assumption \ref{assumption_derivative_process}, it can be seen with Hoelder's inequality that \[
    \|\sup_{u\in[0,1]}|\partial_u^2 g(\tilde Y_t(u))| \|_1 < \infty.
\]
The dominated convergence theorem yields that the last term in \reff{theorem_bias_eq1} is $o(h^2)$, giving the final result.

(ii) By Lemma \ref{lemma_statapprox}(ii), we have in each component that
\[
    |\|\hat G_h(u) - G(u)\|_2 - \|\tilde G_h(u) - G(u)\|_2| \le \|\hat G_h(u) - \tilde G_h(u)\|_2 \le c' (nh)^{-1}.
\]
We furthermore have
\begin{equation}
    \IE|\tilde G_h(u) - G(u)|_2^2 = \IE|\tilde G_h(u) - \IE \tilde G_h(u)|_2^2 + |\IE \tilde G_h(u) - G(u)|_2^2.\label{theorem_bias_eq2}
\end{equation}
Similar as in (i), we obtain
\begin{equation}
    \IE \tilde G_h(u) - G(u) = \frac{h^2}{2}\mu_K \partial_u^2 G(u) + O((nh)^{-1}) + o(h^2).\label{theorem_bias_eq3}
\end{equation}
By Lemma \ref{lemma_variance_calc}, we have
\begin{eqnarray}
    &&\IE|\tilde G_h(u) - \IE \tilde G_h(u)|_2^2
    = \sum_{j=1}^{d}\Var(\tilde G_h(u)_j)\nonumber\\
    &=& \frac{1}{nh}\sigma_K^2 \sum_{j=1}^{d}\sum_{k\in\Z}\Cov(g_j(\tilde Y_0(u)), g_j(\tilde Y_k(u))) + O((nh)^{-2}).\label{theorem_bias_eq4}
\end{eqnarray}
Inserting \reff{theorem_bias_eq4} and \reff{theorem_bias_eq3} into \reff{theorem_bias_eq2}, we obtain the result.

\end{proof}

\subsection{Proof of Theorem \ref{thm_crossval}}

\begin{proof}[Proof of Theorem \ref{thm_crossval}]
	We use the proof techniques from \cite{richterdahlhaus2018aos}. Choosing $\ell(x,y,\theta) = \frac{1}{2}(g(x,y) - \theta)^2$, we find that $\theta \mapsto L(u,\theta) := \IE \ell(\tilde X_0(u),\tilde Y_{-1}(u),\theta)$ is uniquely minimized by $G(u) = \IE g(\tilde Y_0(u))$.\\
	Assumptions 3.1, 3.3, 3.4 and 3.7 easily follow from the assumptions stated at the beginning (weight function $w$..., set of bandwidths $H_n$, $G(u)$ twice continuously diff, conditions on $g$, twice cont. ...). It can be easily seen that Assumption 3.7(3) therein is only needed in the version we ask for on $u \mapsto \partial_u^2 g(\tilde Y_0(u))$.\\
	The only condition which is not fulfilled is Assumption 3.3(4) therein, which asks
	\[
		\nabla_{\theta} \ell(\tilde Y_0(u),\theta)\big|_{\theta = G(u)} = G(u) - g(\tilde Y_0(u))
	\]
	to be an uncorrelated sequence. We now follow the main steps of the proof of Theorem 3.6 in \cite{richterdahlhaus2018aos} and emphasize the steps where the condition of uncorrelatedness has to be circumvented.\\
	By (18) in \cite{richterdahlhaus2018aos}, we have uniformly in $h\in H_n$,
	\begin{eqnarray}
		d_M(h) &:=& \IE \int_0^{1}|\hat G_h(u) - G(u)|_2^2 w(u) \dif u\nonumber\\
		&=& \frac{\mu_{K}}{nh} + \frac{h^4}{4}d_K^2 \int_0^{1}|\partial_u^2 G(u)|_2^2 w(u) \dif u + o((nh)^{-1}) + o(h^4).\label{cvproof3}
	\end{eqnarray}
	Using the same arguments as in \cite{richterdahlhaus2018aos}, equation (26) therein, we have
	\begin{equation}
		\sup_{h\in H_n}\Big|\frac{d_A(h) - d_M(h)}{d_M(h)}\Big| \to 0\quad a.s.\label{cvproof5}
	\end{equation}
	Define
	\[
		d_{A,-}(h) := \frac{1}{n}\sum_{t=1}^{n}|\hat G_h^{-}(t/n) - G(t/n)|_2^2 w(t/n).
	\]
	Then we have 
	\begin{eqnarray*}
		|d_{A}(h) - d_{A,-}(h)| &\le& \frac{1}{n}\sum_{t=1}^{n}|\hat G_h(t/n) - \hat G_h^{-}(t/n)|_2^2 w(t/n)\\
		&&\quad\quad\quad + \Big|\frac{2}{n}\sum_{t=1}^{n}\langle G(t/n) - \hat G_h(t/n), \hat G_h(t/n) - \hat G_h^{-}(t/n)\rangle w(t/n)\Big|\\
		&\le& W_{n,h} + 2 W_{n,h}^{1/2}d_{A}(h)^{1/2},
	\end{eqnarray*}
	where with $k_{n,h}(t) := \frac{1}{n}\sum_{s=1}^{n}K_h^{(n)}((s-t)/n)$,
	\[
		W_{n,h} = \frac{1}{n^3}\sum_{t=1}^{n}\Big(\sum_{s=1}^{n}|g(Y_{s,n})|_1 \{K_h-  \frac{K_h^{(n)}}{k_{n,h}(t)}\}((s-t)/n)\Big)^2w(t/n)^2.
	\]
	Note that $W_{n,h}$ is Lipschitz-continuous in the sense that
	\[
		|W_{n,h} - W_{n,h'}| \le C(n)\cdot |h-h'|
	\]
	with some polynomial $C(n)$ in $n$. This allows us to use chaining arguments to prove uniform convergences in $h$. In the following we use the decomposition
	\[
		W_{n,h} \le 4(W_{n,h,1} + W_{n,h,2} + W_{n,h,3}),
	\]
	where with $c(s_1,s_2) := \sum_{t=1}^{n}\{K_h-  \frac{K_h^{(n)}}{k_{n,h}(s_1)}\}((s_1-t)/n)\cdot \{K_h-  \frac{K_h^{(n)}}{k_{n,h}(s_2)}\}((s_2-t)/n) w(t/n)^2$ and $\IE_0(Z) = Z - \IE Z$,
	\begin{eqnarray*}
		W_{n,h,1} &:=& \frac{1}{n^3}\sum_{t=1}^{n}\Big(\sum_{s=1}^{n}\IE |g(\tilde Y_{s}(s/n))|_1 \{K_h-  \frac{K_h^{(n)}}{k_{n,h}(t)}\}((s-t)/n)\Big)^2 w(t/n)^2,\\
		W_{n,h,2} &:=& \frac{1}{n^3}\sum_{s_1,s_2=1}^{n}\IE_0 |g(\tilde Y_{s_1}(s_1/n))|_1 \cdot \IE_0 |g(\tilde Y_{s_2}(s_2/n))|_1 \cdot c(s_1,s_2) ,\\
		W_{n,h,3} &:=& \frac{1}{n^3}\sum_{t=1}^{n}\Big(\sum_{s=1}^{n}|g(Y_{s,n}) - g(\tilde Y_s(s/n))|_1 \{K_h-  \frac{K_h^{(n)}}{k_{n,h}(t)}\}((s-t)/n)\Big)^2 w(t/n)^2.
	\end{eqnarray*}
	Since for all $q \ge 1$, $\sup_{h\in H_n}(nh)\|W_{n,h,3}\|_q \le O(h)$, it is easy to see by a chaining argument that $\sup_{h \in H_n}(nh)|W_{n,h,3}| \to 0$ a.s.\\
	Note that
	\[
	    k_{n,h}(t) - 1 = O((nh)^{-1} + n^{-\alpha}),
	\]
	and for $t \in \{1,...,n\}$ with $w(t/n) \not= 0$,
	\[
	    \sum_{s=1}^{n}\big|K_h((s-t)/n) - \frac{K_h^{(n)}((s-t)/n}{k_{n,h}(t)}\big| = O(n\cdot (n^{-\alpha} + (nh)^{-1})).
	\]
	By Lemma \ref{lemma_exponential_moments}, we have $\sup_{s,n}\IE|g(\tilde Y_{s}(s/n))|_1 =O(1)$, thus with some constant $c' > 0$,
	\[
		|W_{n,h,1}| \le (c')^2 |K|_{\infty}^2 |w|_{\infty}^2 (n^{-2\alpha} + (nh)^{-2}).
	\]
	Therefore we obtain $\sup_{h\in H_n}(nh)|W_{n,h,2}| \to 0$ if $(nh)\cdot n^{-2\alpha} \to 0$.\\
	Using Lemma 8.1(ii) from the Supplementary material of \cite{richterdahlhaus2018aos}, we have with some constants $c',c'' > 0$,
	\begin{eqnarray*}
		\IE W_{n,h,2} &\le& c'\frac{|w|_{\infty}^2}{n^3}\sup_{k\in\IZ}\sum_{1\le t,t+k\le n}|c(t,t+k)| \le c''\frac{|w|_{\infty}^2|K|_{\infty}^2}{n(nh)^2}\cdot n\cdot (n^{-\alpha} + (nh)^{-1})(nh)\\
		&=& c''|w|_{\infty}^2|K|_{\infty}^2 \cdot (n^{-\alpha} + (nh)^{-1})(nh)^{-1},
	\end{eqnarray*}
	which shows that $\sup_{h\in H_n}(nh)|\IE W_{n,h,2}| \to 0$.\\
	Finally, by the same Lemma we obtain for all $q > 0$ with some constant $c' = c'(q) > 0, c'' = c''(q) > 0$:
	\begin{eqnarray*}
		\|W_{n,h,2} - \IE W_{n,h,2}\|_q &\le& c' \frac{1}{n^3}\Big(\sum_{s_1,s_2=1}^{n}c(s_1,s_2)^2\Big)^{1/2}\\
		&\le& c''|K|_{\infty}^2 \frac{(n^{-\alpha} + (nh)^{-1})(nh)}{n(nh)^2}\cdot ( (n^{-\alpha} + (nh)^{-1})(nh)\cdot n)^{1/2}\\
		&\le&  c'' |K|_{\infty}^2 \frac{(n^{-\alpha} + (nh)^{-1})^{3/2}h^{1/2}}{nh},
	\end{eqnarray*}
	showing with a chaining argument that $\sup_{h\in H_n}(nh)|W_{n,h,2} - \IE W_{n,h,2}| \to 0$ a.s.
	In total, we have seen that
	\[
		\sup_{h\in H_n}(nh)|W_{n,h}| \to 0\quad a.s.,
	\]
	and thus with \reff{cvproof3},
	\begin{equation}
		\sup_{h\in H_n}\frac{|d_A(h) - d_{A,-}(h)|}{d_A(h)} \to 0 \quad a.s.\label{cvproof4}
	\end{equation}
	We now analyze the difference
	\[%\frac{1}{n}\sum_{t=1}^{n}|G(t/n) - \hat G_h^{-}(t/n)|_2^2
		2[H(\hat G_h^{-}) - H(G)] - d_{A,-}(h) =  -\frac{2}{n}\sum_{t=1}^{n}\langle g(Y_{t,n}) - G(t/n), \hat G_h^{-}(t/n)-G(t/n)\rangle =: S_{n,h}
	\]
	Following the proof of Lemma 3.16 in \cite{richterdahlhaus2018aos}, we show
	\begin{equation}
		\sup_{h\in H_n}\Big|\frac{2[H(\hat G_h^{-}) - H(G)]-d_{A,-}(h)}{d_{M}(h)}\Big| \to 0\quad a.s.\label{cvproof1}
	\end{equation}
	Define $\bar G_{h}^{-}(u) := \frac{1}{n}\sum_{t=1}^{n} K^{(n)}_h(t/n-u)\cdot g(\tilde Y_t(u))$. Following the proof of Lemma 2.5 in \cite{richterdahlhaus2018aos} ((41) and the discussion of $R_{0,n}$, $R_{1,n}$, $R_{2,n}$, $R_{3,n}$ therein), we obtain that
	\[
		\tilde S_{n,h} :=  -\frac{2}{n}\sum_{t=1}^{n}\langle g(\tilde Y_{t}(t/n)) - G(t/n), \bar G_h^{-}(t/n)-\IE \bar G_h^{-}(t/n)\rangle 
	\]
	fulfills $\sup_{h\in H_n}|\frac{S_{n,h} - \tilde S_{n,h}}{d_M(h)}| \to 0$ a.s. To use the argument for $R_{3,n}$ therein to obtain \reff{cvproof1}, we have to verify that
	\begin{equation}
		\sup_{h\in H_n}\frac{\IE \tilde S_{n,h}}{d_M(h)} \to 0.\label{cvproof2}
	\end{equation}
	We now cannot argue with the uncorrelatedness of $\nabla_{\theta} \ell(\tilde Y_0(u),\theta)\big|_{\theta = G(u)} = G(u) - g(\tilde Y_0(u))$. Instead we use a direct calculation of $\IE \tilde S_{n,h}$: It holds that $\sup_{u\in[0,1]}\delta_2^{g(\tilde Y(u))}(k) \le c' k^{-\kappa}$ with some $c' > 0$ (cf. also \reff{lemma_emp_eq1}), thus
	\begin{eqnarray*}
		|\IE \tilde S_{n,h}| &\le& \frac{2}{n(nh)}\sum_{s,t=1}^{n}|K^{(n)}((s-t)/(nh)|\cdot |\Cov(g(\tilde Y_t(t/n)), g(\tilde Y_s(t/n)))|\\
		&\le& \frac{2}{n(nh)}\sum_{s,t=1}^{n}|K^{(n)}((s-t)/(nh)|\cdot \sup_{u\in[0,1]}\delta_2^{g(\tilde X(u))}(|s-t|)\\
		&\le& \frac{4c'|K|_{\infty}}{n(nh)}\sum_{t=1}^{n} \sum_{k=\lfloor (1-\varepsilon)n^{-\alpha}(nh)\rfloor}^{\infty} k^{-\kappa}\\
		&=& O(\frac{(n^{-\alpha}(nh))^{-\kappa+1}}{nh})
	\end{eqnarray*}
	which shows \reff{cvproof2} if $n^{-\alpha}(nh) \to \infty$.
	
	Using \reff{cvproof1}, \reff{cvproof4} and \reff{cvproof5}, the result now follows along the same lines as the proof of Theorem 3.6 in \cite{richterdahlhaus2018aos}.
\end{proof}

\subsection{Proof of Theorem \ref{upperbound_thm2}}

\begin{proof}[Proof of Theorem \ref{upperbound_thm2}]
	Define $\Delta_h(u) := \sup_{h' \le h}|G_{h'}(u) - G(u)|_2$. Since $h'\in H_n$, for $n$ large enough it holds that
	\[
		\frac{1}{n}\sum_{t=1}^{n}K_{h'}(t/n-u) \ge \frac{1}{2}.
	\]
	Furthermore, with some constant $c' > 0$ only depending on $|K|_{\infty}$, $|G|_{\infty}$ and the corresponding Lipschitz constants of $K,G$:
	\[
		\Big|\frac{1}{n}\sum_{t=1}^{n}K_{h'}(t/n-u)\{G(t/n) - G(u)\} - \int_0^{1}K_{h'}(v-u)\{G(v) - G(u)\} \dif v\Big|_{\infty} \le c' n^{-1}.
	\]
	Since $G$ is twice continuously differentiable,
	\[
		\int_0^{1}K_{h'}(v-u)\{G(v) - G(u)\} \dif v = \frac{h^2}{2}\mu_K\cdot \partial_u^2 G(u) + o(h^2).
	\]
	We obtain
	\begin{eqnarray*}
		\Delta_h(u) &=& \sup_{h'\le h}\Big(\frac{1}{n}\sum_{t=1}^{n}K_{h'}(t/n-u)\Big)^{-1}\cdot \Big|\frac{1}{n}\sum_{t=1}^{n}K_{h'}(t/n-u)\{G(t/n) - G(u)\}\Big|_2\\
		&\le& 2\cdot \frac{h^2}{2}\mu_K\cdot |\partial_u^2 G(u)|_2 + o(h^2),
	\end{eqnarray*}
	and thus for $h$ small enough, $\Delta_h(u) \le 2h^2 \mu_K\cdot |\partial_u^2 G(u)|_2$. Define
	\begin{eqnarray*}
		h_{opt,n}(u) &:=& \argmin_{h > 0}\Big\{\frac{\log(n)}{nh}\sigma_K^2 \cdot \tr(\Sigma(u)) + \frac{h^4}{4}\mu_K^2\cdot |\partial_u^2 G(u)|_2^2\Big\}\\
		&=& \Big(\frac{4\sigma_K^2 \cdot \tr(\Sigma(u))}{\mu_K^2\cdot |\partial_u^2 G(u)|_2^2}\Big)^{1/5}\cdot \Big(\frac{\log(n)}{n}\Big)^{1/5}.
	\end{eqnarray*}
	By Proposition \ref{upperbound}, we have
	\begin{eqnarray}
		\IE |\hat G_{\hat h(u)}^{\circ}(u) - G(u)|_2^2 &\le& c \cdot v^2(\bar h,u) \lambda(\bar h)^2 + c\sum_{h\in H_n}v^2(h_0(u),u) \lambda(h_0(u))^2\Ii_{\{h_0(u) = a\cdot h\}}\nonumber\\
		&&\quad\quad+ c'\cdot \log(n)^2 n^{-1}.\label{upperbound_thm2_eq2}
	\end{eqnarray}
	For the second summand in \reff{upperbound_thm2_eq2} we find, for $n$ large enough, the upper bound
	\begin{eqnarray*}
		&&\sum_{h\in H_n}v^2(h_0(u),u) \lambda(h_0(u))^2\Ii_{\{h_0(u) = a\cdot h\}}\\
		&\le&\sum_{h\in H_n}\min\{v^2(ah,u) \lambda(ah)^2,\frac{1}{D_1^2}\Delta_h(u)^2\}\\
		&\le& \frac{1}{D_1^2}\sum_{h\in H_n, h \le h_{opt,n}(u)}\Delta_h(u)^2 + \sum_{h\in H_n, h > h_{opt,n}(u)}v^2(ah,u)\lambda(ah)^2\\
		&\le& \frac{4}{D_1^2(1-a^4)}h_{opt,n}(u)^4\mu_K^2|\partial_u^2 G(u)|_2^2 + \frac{D_2}{a(1-a)}\sigma_K^2 \tr(\Sigma(u))\frac{\log(n)}{n h_{opt,n}(u)},
	\end{eqnarray*}
	which together with \reff{upperbound_thm2_eq2} gives the result.
\end{proof}

\begin{proposition}\label{upperbound}
	Suppose that Assumption \ref{def_localstat} holds for all $q > 0$, and that Assumption \ref{assumption_sigma} and Assumption \ref{assumption_dependence} hold. Define $G_h(u) := \IE \hat G_h^{\circ}(u)$, and
	\[
		h_{0}(u) := \sup\{h\in H_n:|G_{h'}(u) - G(u)|_2 \le \frac{C^{\#}}{8} v(h,u) \lambda(h)\quad \text{ for all } h'\in H_n, h' \le h\}. 
	\]
	Suppose that
	\[
		H_n = \{a^{-k}:k\in \IN_0\} \cap [c_H\cdot \log(n)^{5+2\alpha M}, 1]
	\]
	with some constant $c_H > 0$ large enough. Then there exists some universal constant $c = c(a) > 0$ and some constant $c' > 0$ (depending on $M,\chi,C,\rho,D,\alpha,a,K$) such that for $n$ large enough,
	\[
		\IE |\hat G_{\hat h(u)}^{\circ}(u) - G(u)|_2^2 \le c \cdot v^2(h_0(u),u) \lambda(h_0(u))^2 + c'\cdot \log(n)^2 n^{-1}.
	\]
\end{proposition}

\begin{proof}[Proof of Proposition \ref{upperbound}]
	We follow the proof strategy of \cite{lepski2011}. During the proof, we use $c' > 0$ for a  constant only dependent on $M,C,D,\alpha,\rho,K$.
	Put
	\begin{eqnarray*}
		v^2(h,u) &:=& \frac{1}{nh}\int K(x)^2 \dif x \cdot \text{tr}(\Sigma(u)),\\
		v^2(h,h',u) &:=&  \frac{1}{n}\int \{K_h(x) - K_{h'}(x)\}^2 \dif x \cdot \text{tr}(\Sigma(u)).
	\end{eqnarray*}
	Define
	\[
		S(u) := \{\hat h(u) \ge h_0(u)\}.
	\]
	Then
	\begin{equation}
		\IE |\hat G_{\hat h(u)}^{\circ}(u) - G(u)|_2^2 = \IE|\hat G_{\hat h(u)}^{\circ}(u) - G(u)|_2^2 \Ii_{S(u)} + \IE|\hat G_{\hat h(u)}^{\circ}(u) - G(u)|_2^2 \Ii_{S(u)^c}.\label{upperbound_eq1}
	\end{equation}
	\emph{Discussion of the first summand in \reff{upperbound_eq1}:} It holds that
	\begin{eqnarray}
		\IE|\hat G_{\hat h(u)}^{\circ}(u) - G(u)|_2^2 \Ii_{S(u)} &\le& 2\IE|\hat G_{\hat h(u)}^{\circ}(u) - \hat G_{h_0(u)}(u)|_2^2 \Ii_{S(u)} + 2\IE |\hat G_{h_0(u)}^{\circ}(u) - G_{h_0(u)}(u)|_2^2\nonumber\\
		&&\quad\quad+ 2|G_{h_0(u)}(u) - G(u)|_2^2.\label{upperbound_eq2}
	\end{eqnarray}
	By Lemma \ref{lemma_statapprox} and Lemma \ref{lemma_variance_calc} and since $ \frac{1}{n}\sum_{t=1}^{n}K_h(t/n-u) \ge \frac{1}{2}$ for $n$ large enough, we have
	\begin{eqnarray}
		\IE |\hat G_{h_0(u)}^{\circ}(u) - G_{h_0(u)}(u)|_2^2 &\le& 8 \IE|\hat G_{h_0(u)}(u) - \tilde G_{h_0(u)}(u)|_2^2 + 8\IE|\tilde G_{h_0(u)}(u) - G_{h_0(u)}(u)|_2^2\nonumber\\
		&\le& 8v^2(h_{0}(u)) + c \cdot n^{-1}.\label{upperbound_eq2_5}
	\end{eqnarray}
	By definition of $\hat h(u)$ and monotonicity of $\lambda(\cdot)$,
	\begin{equation}
		\IE|\hat G_{\hat h(u)}^{\circ}(u) - \hat G_{h_0(u)}^{\circ}(u)|_2^2 \Ii_{S(u)} \le (C^{\#})^2\IE[\hat v^2(h_0(u),u)]\lambda(h_0(u))^2.\label{upperbound_eq3}
	\end{equation}
	We now discuss $\IE[\hat v^2(h_0(u),u)]$. It holds that
	\begin{eqnarray*}
		\IE[\hat v^2(h_0(u),u)\Ii_{\{\tr(\hat \Sigma_n(u)) > 2 \tr(\Sigma(u))\}}] &\le& \|\hat v^2(h_0(u),u)\|_2 \cdot \IP\big(|\tr(\hat \Sigma_n(u)) - \tr(\Sigma(u))| > \tr(\Sigma(u))\big)^{1/2}\\
		&\le& v^2(h_0(u),u)
	\end{eqnarray*}
	for $n$ large enough due to Assumption \ref{assumption_sigma}. We therefore have
	\begin{equation}
		\IE \hat v^2(h_0(u),u) \le 2v^2(h_0(u),u) + \IE[\hat v^2(h_0(u),u)\Ii_{\{\tr(\hat \Sigma_n(u)) > 2 \tr(\Sigma(u))\}}] \le 3 v^2(h_0(u),u).\label{upperbound_eq4}
	\end{equation}
	Using  \reff{upperbound_eq3} and \reff{upperbound_eq4}, we obtain
	\begin{equation}
		\IE|\hat G_{\hat h(u)}^{\circ}(u) - \hat G_{h_0(u)}^{\circ}(u)|_2^2\Ii_{S(u)} \le 3(C^{\#})^2v^2(h_0(u),u)\lambda(h_0(u))^2.\label{upperbound_eq5}
	\end{equation}
	By definition of $h_0(u)$, we have
	\begin{equation}
		|G_{h_0(u)}(u) - G(u)|_2 \le \frac{C^{\#}}{8}v(h_0(u),u)\lambda(h_0(u)).\label{upperbound_eq6}
	\end{equation}
	Inserting \reff{upperbound_eq2_5}, \reff{upperbound_eq5} and \reff{upperbound_eq6} into \reff{upperbound_eq2}, we obtain
	\begin{equation}
		\IE|\hat G_{\hat h(u)}^{\circ}(u) - G(u)|_2^2 \Ii_{S(u)} \le \big[8 + (\frac{C^{\#}}{8})^2 + 3(C^{\#})^2\big] v^2(h_0(u),u)\lambda(h_0(u))^2 + c'\cdot n^{-1}.\label{upperbound_eq24}
	\end{equation}
	\emph{Discussion of the second summand in \reff{upperbound_eq1}:} Let $H_n(h) := \{h' \in H_n: h' < h\}$. By definition of $h_0(u)$ and by monotonicity of $v(\cdot)$, $\lambda(\cdot)$, we obtain for $h' \le h \le h_0(u)$:
	\begin{equation}
		|G_{h'}(u) - G(u)|_2 \le \frac{C^{\#}}{8}v(h_0(u),u)\lambda(h_0(u)) \le \frac{C^{\#}}{8}v(h,u)\lambda(h).\label{upperbound_eq10}
	\end{equation}
	Decompose
	\[
		S(u)^c = \bigcup_{h\in H_n(a\cdot h_0(u))}\bigcup_{h'\in H_n(h)}E(h,h',u),\quad\quad E(h,h',u) := \{|\hat G_h^{\circ}(u) - \hat G_{h'}^{\circ}(u)|_2 > C^{\#}\hat v(h',u)\lambda(h')\}.
	\]
	Let
	\[
		A_1 := \{\text{tr}(\hat \Sigma_n(u)) \ge \frac{1}{2} \text{tr}(\Sigma(u))\},
	\]
	and define
	\begin{eqnarray*}
		N(h,h',u) &:=& (\hat G_h^{\circ}(u) - G_h(u)) - (\hat G_{h'}^{\circ}(u) - G_{h'}(u)),\\
		\tilde N(h,h',u) &:=& (\tilde G_h^{\circ}(u) - G_h(u) - (\tilde G_{h'}^{\circ}(u) - G_{h'}(u)),\\
		N(h,u) &:=& \hat G_h^{\circ}(u) - G_h(u),\\
		\tilde N(h,u) &:=& \tilde G_h^{\circ}(u) - G_h(u).
	\end{eqnarray*}
	We have
	\begin{eqnarray*}
		E(h,h',u) \cap A_1 &\subset& \{|\hat G_h^{\circ}(u) - \hat G_{h'}^{\circ}(u)|_2 > \frac{C^{\#}}{2}v(h',u)\lambda(h')\}\\
		&\subset& \{\frac{2C^{\#}}{8}v(h',u)\lambda(h') + |N(h,h',u)|_2 > \frac{C^{\#}}{2}v(h',u)\lambda(h')\}\\
		&\subset& \{|N(h,h',u)|_2 > C^{\#}(\frac{1}{2} - \frac{2}{8})v(h',u)\lambda(h')\} =: E_0(h,h',u).
	\end{eqnarray*}
	We conclude with the Cauchy-Schwarz inequality, \reff{upperbound_eq10} and Assumption \ref{assumption_sigma}:
	\begin{eqnarray*}
		&& \IE|\hat G_{\hat h(u)}^{\circ}(u) - G(u)|_2^2 \Ii_{S(u)^c}\\
		&\le& \sum_{h\in H_n(a h_0(u))}\sum_{h' \in H_n(h)}\IE|\hat G_{h}^{\circ}(u) - G(u)|_2^2 \Ii_{E(h,h',u)}\\
		&\le& \sum_{h\in H_n(a h_0(u))}\sum_{h' \in H_n(h)}\IE|\hat G_{h}^{\circ}(u) - G(u)|_2^2 \Ii_{E_0(h,h',u)} + c' \log(n)^2 n^{-1}\\
		&\le& \sum_{h\in H_n(a h_0(u))}\sum_{h' \in H_n(h)}\IE\big[ \big(|N(h,u)|_2 + \frac{C^{\#}}{8}v(h,u)\lambda(h_0(u))\big)^2\big] \Ii_{E_0(h,h',u)} + c' \log(n)^2 n^{-1}.
	\end{eqnarray*}
	With Lemma \ref{lemma_statapprox}, we can replace $|N(h,u)|_2$ with $|\tilde N(h,u)|_2$ with  error $O(\log(n)^2 n^{-1})$ due to
	\[
		\IE\big[ (|N(h,u)|_2 - |\tilde N(h,u)|_2)^2\big] \le |\ \|N(h,u) - \tilde N(h,u)\|_2\ |_2^2 \le c'n^{-2}.
	\]
	Similarly, the set $A(h,h') := \{|\ |N(h,h',u)|_2 - |\tilde N(h,h',u)|_2| \le \frac{C^{\#}}{2}(\frac{1}{2}-\frac{2}{8})v(h_0(u),u)\lambda(h_0(u))\}$ has the property
	\[
		\IP(A(h,h')) \le \frac{c' n^{-4}}{(v(h_0(u),u)\lambda(h_0(u)))^4} =O(n^{-2}),
	\]
	allowing to replace $|N(h,h',u)|_2$ by $|\tilde N(h,h',u)|_2$ in $E_0(h,h',u)$ with replacement error $O(\log(n)^2 n^{-1})$. Together with $v(h,h',u) \le v(h',u)$, we have shown that
	\[
		\tilde E_0(h,h',u) := \{|\tilde N(h,h',u)|_2 > \frac{C^{\#}}{2}(\frac{1}{2} - \frac{2}{8})v(h,h',u)\lambda(h')\},
	\]
	fulfills
	\begin{eqnarray}
		 \IE|\hat G_{\hat h(u)}(u) - G(u)|_2^2 \Ii_{S(u)^c} &\le& \sum_{h\in H_n(a h_0(u))}\sum_{h' \in H_n(h)}\IE\big[|\tilde N(h,u)|_2 + \frac{C^{\#}}{8}v(h,u)\lambda(h_0(u))\big]^2 \Ii_{\tilde E_0(h,h',u)}\nonumber\\
		 &&\quad\quad\quad\quad\quad\quad\quad\quad\quad\quad + c' \log(n)^2 n^{-1}.\label{upperbound_eq12}
	\end{eqnarray}
	%((Braucht man ueberhaupt noch $v(h)$ in dem Event $\tilde E_0$?))\\
	
	Put $x := \frac{C^{\#}}{8}\lambda(h_0(u))$ and $D^{\#} := \frac{C^{\#}}{2}(\frac{1}{2}-\frac{2}{8})$. We now discuss the summands in \reff{upperbound_eq12}. It holds that
	\begin{eqnarray}
		&&\IE\big[x + v(h,u)^{-1}|\tilde N(h,u)|_2\big]^2 \Ii_{\{|\tilde N(h,h',u)|_2 > D^{\#}v(h,h',u)\lambda(h')\}}\nonumber\\
		&=& 2\int_0^{\infty}z\cdot \IP(v(h,u)^{-1}|\tilde N(h,u)|_2 > z-x,|\tilde N(h,h',u)|_2 > D^{\#}v(h,h',u)\lambda(h')) \dif z\nonumber\\
		&\le& 2 \int_0^{x + \sqrt{2^7\log(n)}}z\cdot \IP(|\tilde N(h,h',u)|_2 > D^{\#}v(h,h',u)\lambda(h')) \dif z\nonumber\\
		&&\quad\quad + 2 \int_{x+\sqrt{2^7\log(n)}}^{\infty}z\cdot \IP(v(h)^{-1}|\tilde N(h,u)|_2 > z-x) \dif z\nonumber\\
		&=& (x+\sqrt{2^7\log(n)})^2 \IP(|\tilde N(h,h',u)|_2 > D^{\#}v(h,h',u)\lambda(h'))\nonumber\\
		&&\quad\quad + 2\int_{\sqrt{2^7\log(n)}}^{\infty}(x+y)\IP(|\tilde N(h,u)|_2 > v(h,u) y) \dif y.\label{upperbound_eq13}
	\end{eqnarray}
	\emph{Discussion of the first summand in \reff{upperbound_eq13}.} Put $v_j^2(h,u) := \frac{1}{nh}\int K(x)^2 \dif x \cdot \Sigma(u)_{jj}$ and $v_j^2(h,h',u) := \frac{1}{n}\int \{K_h(x) - K_{h'}(x)\}^2 \dif x \cdot \Sigma(u)_{jj}$. Then, with Lemma \ref{lemma_bernstein_kerne2},
	\begin{eqnarray*}
		&&\IP(|\tilde N(h,h',u)|_2 > D^{\#}v(h,h',u)\lambda(h'))\\
		&\le& \IP\Big(\sum_{j=1}^{d}\tilde N(h,h',u)_j^2 > (D^{\#})^2\sum_{j=1}^{d}v_j^2(h,h',u)\lambda(h')^2\Big)\\
		&\le& \sum_{j=1}^{d}\IP\big((nh')|\tilde N(h,h',u)_j|  > D^{\#}(nh')v_j(h,h',u)\lambda(h')\big)\\
		&\le& d\cdot \sup_{j=1,...,d}\Big\{2\exp\Big(-\frac{(D^{\#})^2\lambda(h')^2}{32 + c_4(\frac{D^{\#}a_n}{(nh')v_j(h,h',u)})^{1/3}\lambda(h')^{5/3}(D^{\#})^{5/3}}\Big) + c_5\frac{n^{-1}}{D^{\#}(nh')v_j(h,h',u)\lambda(h')}\Big\}.
	\end{eqnarray*}
By Lemma \ref{lemma_variance_calc}(ii), $c_4(\frac{a_n}{(nh')v_j(h,h',u)})^{1/3}\lambda(h')^{5/3}(D^{\#})^{5/3} \le 32$ is fulfilled if
\begin{eqnarray*}%%%soll groesser $a_n$ sein.
	\big(\frac{32}{c_4}\big)^3\lambda(h')^{-5} \cdot (nh')v_j(h,h',u) &\ge& \big(\frac{32}{c_4}\big)^3 \frac{1}{D_2^{5/2}}\cdot \big[c_6 (a-1)^2 \inf_{j=1,...,d}(\Sigma(u)_{jj})\big]^{1/2}\cdot (\frac{nh'}{\log(n)^5})^{1/2}\\
	&\ge& 2a_n = 2\tilde c_1 \cdot 8^{1/\tau_2} \log(n)^{1/\tau_2},
\end{eqnarray*}
i.e. if $h' \ge c' \cdot \log(n)^{5+\frac{2}{\tau_2}}\cdot n^{-1}$ for $c' > 0$ large enough (which is fulfilled due to $h'\in H_n$).
We obtain that for $h\in H_n$, $h' \in H_n(h)$ with $D^{\#} = 8$:
\begin{equation}
	\IP(|\tilde N(h,h',u)|_2 > v(h,h',u)\lambda(h')) \le 2d\cdot \frac{h'}{\bar h} +  \frac{c_5}{D^{\#}c_6^{1/2} |a-1|\inf_{j=1,...,d}(\Sigma(u)_{jj})^{1/2}}\cdot n^{-1}.\label{upperbound_eq20}
\end{equation}
	\emph{Discussion of the second summand in \reff{upperbound_eq13}.} We have by Lemma \ref{lemma_bernstein_kerne} that
	\begin{eqnarray*}
		&&\IP(|\tilde N(h,u)|_2 > v(h,u)y)\\
		&\le& \sum_{j=1}^{d}\IP((nh)|\tilde N(h,u)_j| > (nh)v_j(h,u)y)\\
		&\le& d \sup_{j=1,...,d}\Big\{2\exp\Big(-\frac{y^2}{32 + c_4 (\frac{a_n}{(nh)v_j(h,u)})^{1/3} y^{5/3}}\Big) + c_5 \frac{n^{-1}}{(nh)v_j(h,u)y}\Big\}.
	\end{eqnarray*}
	We have with $c_n := \frac{a_n}{(nh)v_j(h,u)}$:
	\[
		\exp\Big(-\frac{y^2}{32 + c_4 (\frac{a_n}{(nh)v_j(h,u)})^{1/3} y^{5/3}}\Big) \le \begin{cases}
			\exp(-\frac{y^2}{64}), & y \le (\frac{32}{c_4})^{3/5}c_n^{-1/5} =: d_n,\\
			\exp\big(-\frac{1}{2c_4}(\frac{y}{c_n})^{1/3}\big), & y > d_n
		\end{cases}
	\]
	Thus
	\begin{eqnarray}
		&& \int_{\sqrt{2^7\log(n)}}^{\infty}(x+y)\IP(|\tilde N(h,u)|_2 > v(h,u) y) \dif y\nonumber\\
		&\le& 4d\sup_j\big[\int_{\sqrt{2^7\log(n)}}^{d_n}(x+y)\exp(-\frac{y^2}{2^6}) \dif y + \int_{d_n}^{\infty}(x+y) \exp\Big(-\frac{1}{2c_4}(\frac{y}{c_n})^{1/3}\Big) \dif y\nonumber\\
		&&\quad\quad\quad\quad\quad\quad + \frac{c_5}{\sigma_K\cdot \inf_{j=1,...,d}(\Sigma(u)_{jj})^{1/2}} n^{-1}\big].\label{upperbound_eq15}
	\end{eqnarray}
	We now discuss the first two summands in \reff{upperbound_eq15}. We have
	\[
		\int_{\sqrt{2^7\log(n)}}^{d_n}(x+y)\exp(-\frac{y^2}{2^6}) \dif y \le [x+d_n] n^{-2},
	\]
	and, with some constant $\tilde c_4 > 0$ only depending on $c_4$,
	\begin{eqnarray*}
		&&\int_{d_n}^{\infty}(x+y)\exp(-\frac{1}{2c_4}(\frac{y}{c_n})^{1/3}) \dif y\\
		&=& c_n \int_{d_n/c_n}^{\infty}(x+c_n z)\exp(-\frac{1}{2c_4}z^{1/3}) \dif z\\
		&\le& \tilde c_4 \cdot c_n\cdot \exp\Big(-\frac{1}{2c_4}\Big(\frac{d_n}{c_n}\Big)^{1/3}\Big)\cdot \Big[x\cdot \big((\frac{d_n}{c_n})^{2/3}+1\big)) + c_n\big((\frac{d_n}{c_n})^{5/3} + 1\big)\Big].
	\end{eqnarray*}
	Here,
	\begin{equation}
		\frac{1}{2c_4}(\frac{d_n}{c_n})^{1/3} = \frac{1}{2}(\frac{16}{c_4^6})^{1/5}c_n^{-2/5} =  \frac{1}{2}\big[\frac{16\int K(x)^2 \dif x \cdot \Sigma(u)_{jj}}{c_4^6}\cdot\frac{nh}{a_n^2}\big]^{1/5}.\label{upperbound_eq16}
	\end{equation}
	We conclude that \reff{upperbound_eq16} is $\ge 2\log(n)$ if $h \ge c' \cdot \log(n)^{5+\frac{2}{\tau_2}}\cdot n^{-1}$ for $c' > 0$ large enough (which is fulfilled due to $h\in H_n$). Clearly, \reff{upperbound_eq16} is $\le O(n^{1/5})$. Summarizing these results into \reff{upperbound_eq15}, we obtain for all $h \in H_n$ that
	\begin{equation}
		\int_{\sqrt{2^7\log(n)}}^{\infty}(x+y)\IP(|\tilde N(h,u)|_2 > v(h,u) y) \dif y \le c' n^{-1}\label{upperbound_eq21}
	\end{equation}
	with some constant $c' > 0$.
	
	Inserting \reff{upperbound_eq20} and \reff{upperbound_eq21} into \reff{upperbound_eq13} and \reff{upperbound_eq12}, we obtain with $\sum_{h'\in H_n(h)}h' \le \frac{h}{1-a}$:
	\begin{eqnarray}
		&& \IE|\hat G_{\hat h(u)}(u) - G(u)|_2^2 \Ii_{S(u)^c}\nonumber\\
		&\le& \frac{2d}{\bar h}\cdot \sum_{h\in H_n(a h_0(u))}\sum_{h'\in H_n(h)}v^2(h,u)\cdot h' + c'\log(n)^2 n^{-1}\nonumber\\
		&\le& \frac{\int K^2(x) \dif x \cdot \tr(\Sigma(u))}{1-a}\cdot \frac{2d}{\bar h}\cdot \sum_{h\in H_n(a h_0(u))} n^{-1} + c'\log(n)^2 n^{-1}\nonumber\\
		&\le& c' \cdot \log(n)^2 n^{-1}.\label{upperbound_eq23}
	\end{eqnarray}
	By \reff{upperbound_eq24} and \reff{upperbound_eq23}, the result follows.
\end{proof}

\begin{lemma}\label{lemma_exponential_moments}
Assume that $g \in \sH(M,\chi,C)$. Suppose that Assumption \ref{assumption_dependence} holds. Put $\tau_2 = (\alpha M)^{-1}$. Then there exist constants $\tilde c_1,\tilde c_2 > 0$ only depending on $M,\chi,C,D$ such that for $q \ge 2$:
	\begin{eqnarray*}
		\big\|g(\tilde Y_t(u))\big\|_q &\le& \tilde c_1 N_{\alpha}(qM)^{M},\\
		\IE \exp\Big[\frac{1}{2}\Big(\frac{|g(\tilde Y_t(u))|}{\tilde c_1}\Big)^{\tau_2}\Big] &\le& \tilde c_2,
	\end{eqnarray*}
\end{lemma}
\begin{proof}[Proof of Lemma \ref{lemma_exponential_moments}:] (i) It holds that
	\begin{eqnarray*}
		\| g(\tilde Y_t(u)) - g(0)\|_q &\le& C \sum_{j=1}^{\infty}\chi_j\|\tilde X_{t-j+1}(u)\|_{qM}\cdot\big(1 + |\chi|_1^{M-1} \|\tilde X_t(u)\|_{qM}^{M-1}\big)\\
		&\le& C |\chi|_1 D N_{\alpha}(qM) \cdot \big( 1+ |\chi|_1^{M-1}D^{M-1} N_{\alpha}(qM)^{M-1}\big).
	\end{eqnarray*}
	Since $|g(0)| \le C$, we obtain $\| g(\tilde Y_t(u))\|_q \le \tilde c_1 N_{\alpha}(qM)^{M}$ with some $\tilde c_1$ only depending on $M,\chi,C,D$.
	
	(ii) Define $\lambda = (2 \tilde c_1^{\tau_2})^{-1}$. By a series expansion of $\exp$, we have
	\[
		\IE\exp\big(\lambda |g(\tilde Y_t(u))|^{\tau_2}\big) = \sum_{q=0}^{\infty}\frac{\lambda^{q}\| g(\tilde Y_t(u))\|_{\tau_2 q}^{\tau_2 q}}{q!}.
	\]
	If $\tau_2 q \ge 2$, we have
	\[
		\| g(\tilde Y_t(u)) \|_{\tau_2 q}^{\tau_2 q} \le \tilde c_1^{\tau_2 q}\cdot \Gamma(\alpha q \tau_2 M+2) = \tilde c_1^{\tau_2 q} \Gamma(q + 2).
	\]
	This shows $\sum_{\tau_2 q \ge 2}\frac{\lambda^q \| g(\tilde Y_t(u)) \|_{\tau_2 q}^{\tau_2 q}}{q!} \le \sum_{\tau_2 q \ge 0}^{\infty}(\lambda \tilde c_1^{\tau_2})^q\cdot \frac{\Gamma(q+2)}{\Gamma(q+1)} = \sum_{\tau_2 q \ge 2}\frac{q+1}{2^q} \le 4$.\\
	In the case $\tau_2 q < 2$, we have
	\[
		\| g(\tilde Y_t(u)) \|_{\tau_2 q}^{\tau_2 q} \le \| g(\tilde Y_t(u)) \|_2^{\tau_2 q} \le  \tilde c_1^{\tau_2 q}\Gamma(2\alpha M+2)^{\tau_2q/2} \le \tilde c_1^{\tau_2 q} \Gamma(2\alpha M + 2).
	\]
	This shows $\sum_{\tau_2 q < 2}\frac{\lambda^q \| g(\tilde Y_t(u)) \|_{\tau_2 q}^{\tau_2 q}}{q!} \le \Gamma(2\alpha M+2)\sum_{q=0}^{\infty}\frac{2^{-q}}{q!} = \exp(2^{-1})\Gamma(2\alpha M + 2)$. The result is obtained with $\tilde c_{2} := 4 + \exp(2^{-1})\Gamma(2\alpha M + 2)$.
\end{proof}

\begin{lemma}[Exponential inequality]\label{lemma_exponential}
	Assume that $\phi:[0,1] \to \IR$ is some measurable function, and $g\in \sH(M,\chi,C)$. Suppose that Assumption \ref{assumption_dependence} holds. Define
	\[
		F_n(\phi,g) := \frac{1}{n}\sum_{t=1}^{n}\phi(t/n)\cdot \{g(\tilde Y_t(t/n)) - \IE g(\tilde Y_t(t/n))\}.
	\]
	Put $\tau = \tau(\alpha,M) := (\frac{1}{2}+\alpha M)^{-1}$. Then there exist constants $c_1,c_2 > 0$ only depending on $M,\chi,C,D$ such that
	\begin{enumerate}
		\item[(i)] \[
		\|F_n(\phi,g)\|_q \le c_1(q-1)^{1/2}n^{-1/2}\Big(\frac{1}{n}\sum_{t=1}^{n}\phi(t/n)^2\Big)^{1/2}\cdot N_{\alpha}(qM)^M,
		\]
		\item[(ii)] \[
			\IP(|F_n(\phi,g)| > \gamma) \le c_2\exp\Big[-\frac{1}{4e}\Big(\frac{\sqrt{n}\cdot \gamma}{c_1 (\frac{1}{n}\sum_{t=1}^{n}\phi(t/n)^2)^{1/2}}\Big)^{\tau}\Big].
		\]
	\end{enumerate}
\end{lemma}
\begin{proof}[Proof of Lemma \ref{lemma_exponential}] 
	 (i) Let $\delta(k) := D\rho^k$. By Hoelder's inequality, we have with some constant $\tilde c$ only dependent on $M,\chi,C,D$:
\begin{eqnarray}
	&& \big\|g(\tilde Y_t(u)) - g(\tilde Y_t^{*}(u))\big\|_q\nonumber\\
	&\le& C \textstyle\sum_{j=1}^{t}\chi_j \| \tilde X_{t-j+1}(u) - \tilde X_{t-j+1}^{*}(u)\|_{qM} \cdot \big(1 + 2 |\chi|_1^{M-1}D^{M-1} \cdot N_{\alpha}(qM)^{M-1}\big)\nonumber\\
	&\le& \tilde c\textstyle\sum_{j=1}^{t}\chi_j \delta(t-j+1) \cdot N_{\alpha}(qM)^{M}.\label{lemma_emp_eq1}                                                                                                                                                           
\end{eqnarray}
Let $\xi(t) := \sum_{j=1}^{t}\chi_j\cdot \delta(t-j+1)$. Obviously, $\sum_{t=1}^{\infty}\xi(t) = \sum_{j=1}^{\infty}\chi_j\sum_{t=j}^{\infty}\delta(t-j+1) < |\chi|_1 |\delta(\cdot)|_1$. We have shown that the dependence measure fulfills $\delta^{g(\tilde Y(u))}_q(k) \le \tilde c \cdot \xi(k)\cdot N_{\alpha}(qM)^{M}$ and is absolutely summable.

By Theorem 2.1 from \cite{rio2009} for $q > 2$ (and for $q = 2$ directly by calculating the variance of the following term), we have
\begin{eqnarray*}
	\| F_{n}(\phi,g) \|_q &\le& \Big\|\frac{1}{n}\sum_{t=1}^{n}\phi(t/n)\big\{g(\tilde Y_t(t/n)) - \IE g(\tilde Y_t(t/n))\big\}\Big\|_q\\
	&\le& \frac{1}{n}\sum_{k=0}^{\infty}\Big\|\sum_{t=1}^{n}\phi(t/n) P_{t-k}g(\tilde Y_t(t/n))\Big\|_q\\
	&\le& \frac{1}{n}\sum_{k=0}^{\infty}(q-1)^{1/2}\Big|\Big(\sum_{t=1}^{n}\phi(t/n)^2 \|P_{t-k}g(\tilde Y_t(t/n))\|_q^{2}\Big)^{1/2}\Big|_2\\
	&\le& (q-1)^{1/2} n^{-1/2}\Big(\frac{1}{n}\sum_{t=1}^{n}\phi(t/n)^2\Big)^{1/2}\cdot \tilde c \sum_{k=0}^{\infty}\xi(k)\cdot N_{\alpha}(qM)^{M}.
\end{eqnarray*}
(ii) Define $Z_n := \tilde cn^{-1/2}\big(\frac{1}{n}\sum_{t=1}^{n}\phi(t/n)^2\big)^{1/2}\cdot \sum_{k=0}^{\infty}\xi(k)$. By Stirling's formula, we have for all $x \ge 1$:
\[
	\sqrt{2\pi}x^{x-\frac{1}{2}}e^{-x} \le \Gamma(x) \le e^{1/12}\cdot \sqrt{2\pi}x^{x-\frac{1}{2}}e^{-x}.
\]
By Markov's inequality, we have for $\gamma,\lambda > 0$:
\begin{eqnarray*}
	\IP(|F_{n}(\phi,g)| \ge \gamma) \le e^{-\lambda \gamma^{\tau}}\IE[e^{\lambda |F_{n}(\phi,g)|^{\tau}}] = e^{-\lambda\gamma^{\tau}}\sum_{q=0}^{\infty}\frac{\lambda^q \||F_{n}(\phi,g)|_2\|_{\tau q}^{\tau q}}{q!}.
\end{eqnarray*}
In the case $\tau q \ge 2$, we have 
\[
	\frac{\lambda^q \|F_{n}(\phi,g)\|_{\tau q}^{\tau q}}{q!} \le \frac{\lambda^q}{\Gamma(q+1)}(\tau q)^{\frac{\tau q}{2}}D(u)^{\tau q}\cdot \Gamma(\alpha M \tau q + 2).
\]
Note that $\alpha M \tau \le 1$ and $\tau(\alpha M + \frac{1}{2}) = 1$, thus
\begin{eqnarray*}
	q^{\frac{\tau q}{2}}\frac{\Gamma(\alpha M \tau q + 2)}{\Gamma(q+1)} &\le& (q+2)^{\frac{\tau q}{2}}\cdot \frac{(\alpha M \tau q + 2)^{\alpha M \tau q + \frac{3}{2}} e^{-(\alpha M \tau q+2)} e^{1/12}}{ (q+1)^{q+\frac{1}{2}} e^{-(q+1)}}\\
	&=& e^{1/12}(q+2)\cdot \Big(\frac{q+2}{q+1}\Big)^{q+\frac{1}{2}} e^{-1} e^{q(1-\alpha M \tau)}\\
	&\le& e^{1/12} (q+2) e^{q}.
\end{eqnarray*}
Define $\lambda := (4e)^{-1} Z_n^{-\tau}$. Since $\tau \le 2$, it holds that $\tau^{\tau/2} \le 2$. Thus
\[
	\sum_{q \ge 2/\tau}\frac{\lambda^q \|F_{n}(\phi,g)\|_{\tau q}^{\tau q}}{q!} \le e^{1/12}\cdot \sum_{q \ge 2/\tau}(q+2)(\lambda \cdot 2e Z_n^{\tau})^q \le e^{1/12}\sum_{q \ge 2/\tau}\frac{q+2}{2^q} \le 4e^{1/12}.
\]
In the case $\tau q < 2$, we have
\begin{eqnarray*}
	\frac{\lambda^q \|F_{n}(\phi,g)\|_{\tau q}^{\tau q}}{q!} &\le& \frac{\lambda^q \||E_{n,b}(g,u)|_2\|_{2}^{\tau q}}{q!} \le  \frac{\lambda^q}{q!}Z_n^{\tau q}\cdot \Gamma(2\alpha M + 2)^{\frac{\tau q}{2}}\\
	&\le& \frac{(4e)^{-q}}{q!}\cdot \Gamma(2\alpha M+2),
\end{eqnarray*}
thus $\sum_{q < 2/\tau}\frac{\lambda^q \|F_{n}(\phi,g)\|_{\tau q}^{\tau q}}{q!} \le \exp((4e)^{-1})\Gamma(2\alpha M + 2)$. So the result is obtained with $c_2 := 4e^{1/2} + \exp((4e)^{-1})\Gamma(2\alpha M + 2)$ and $c_1 = \tilde c \sum_{k=0}^{\infty}\xi(k) = \tilde c |\chi|_1 |\delta(\cdot)|_1$.
\end{proof}

\begin{lemma}[Bernstein inequality]\label{lemma_bernstein}
	Assume that $g\in \sH(M,\chi,C)$ and that Assumption \ref{assumption_dependence} holds. Let $\phi:[0,1] \to \R$ be a measurable function.
	Define
	\[
		W_n := \sum_{t=1}^{n}\phi(t/n)\{g(\tilde Y_t(t/n)) - \IE g(\tilde Y_t(t/n))\}.
	\]
	Assume that $s_n := \#\{t\in \{1,...,n\}: \phi(t/n) \not= 0\}$ fulfills
	\begin{equation}
		\Var(W_n) \ge \text{const.}(M,\chi,C,D,\rho,\phi)\cdot s_n.\label{lemma_bernstein_condition1}
	\end{equation}
	Then there exist some constants $c_4,c_5 > 0$ only dependent on $M,\chi,C,D,\rho,|\phi|_{\infty}$ such that:
	\[
		\IP\Big(\big|\sum_{t=1}^{n}\phi(t/n) \{g(\tilde Y_t(t/n)) - \IE g(\tilde Y_t(t/n))\}\big| > \gamma\Big) \le 2\exp\Big(-\frac{\gamma^2}{16 \Var(W_n) + c_4 a_n^{1/3} \gamma^{5/3}}\Big) + c_5 \frac{n^{-1}}{\gamma}.
	\]
		%\IP\Big(\big|\sum_{t=1}^{n}\phi(t/n) \{g(\tilde Y_t(t/n)) - \IE g(\tilde Y_t(t/n))\}\big| > \gamma\Big) \le 2\exp\Big(-\frac{\gamma^2}{16(nh)^2 \sigma_{n,h}^2 + c_4 a_n^{1/3} \gamma^{5/3}}\Big) + c_5 \frac{n^{-1}}{\gamma}.
	with $a_n := \tilde c_1(8\log(n))^{1/\tau_2}$ ($\tilde c_1, \tau_2$ from Lemma \ref{lemma_exponential_moments}).
	%\begin{eqnarray*}
		%\sigma_{n,h}^2 &:=& \frac{1}{nh}\int_0^{1}K(x)^2 \dif x \cdot \sum_{k\in\Z}\Cov(g(\tilde Y_0(u)), g(\tilde Y_k(u))),\\
		%\sigma_{n,h,h'}^2 &:=& \frac{1}{n}\int_0^{1}\{K_h(x) - K_{h'}(x)\}^2 \dif x \cdot \sum_{k\in\Z}\Cov(g(\tilde Y_0(u)), g(\tilde Y_k(u))).
	%\end{eqnarray*}
\end{lemma}
\begin{proof} (i) \emph{Step 1: Truncation.} Define $W_n^{\circ} := \sum_{t=1}^{n}Z_t^{\circ}$, where $Z_t^{\circ} = Z_t^{\circ}(t/n)$, $\Psi(x) = x\Ii_{\{|x| \le a_n\}} + a_n \Ii_{\{|x| > a_n\}}$,
\[
	Z_t^{\circ}(u) := \phi(u)\cdot [\Psi(g(\tilde Y_t(u))) - \IE \Psi(g(\tilde Y_t(u)))].
\]
We have
\begin{eqnarray}
	\|W_n - W_n^{\circ}\|_q &\le& 2\sum_{t=1}^{n}|\phi(t/n)|\cdot \|g(\tilde Y_t(t/n))\Ii_{\{|g(\tilde Y_t(t/n))| > a_n\}}\|_q\nonumber\\
	&\le& 2|\phi|_{\infty}n\cdot \sup_{u\in[0,1]}\|g(\tilde Y_0(u))\|_{2q}\cdot \sup_{u\in[0,1]}\IP(|g(\tilde Y_0(u))|>a_n)^{1/2}\label{proof_lemma_bernstein_eq0}
\end{eqnarray}
By Lemma \ref{lemma_exponential_moments}, we have
\[
	\IP(|g(\tilde Y_0(u))|>a_n) \le \tilde c_2\cdot \exp\Big(-\frac{1}{2}\Big(\frac{a_n}{\tilde c_1}\Big)^{\tau_2}\Big) \le \tilde c_2 n^{-4}.
\]
Inserting this into \reff{proof_lemma_bernstein_eq0} and using Lemma \ref{lemma_exponential_moments} to bound $\sup_{u\in[0,1]}\|g(\tilde Y_0(u))\|_{2q} \le \tilde c_1 N_{\alpha}(2qM)^M$, we obtain.
\begin{equation}
	\|W_n - W_n^{\circ}\|_q \le 2|\phi|_{\infty}\tilde c_1 N_{\alpha}(2qM)^M\sqrt{\tilde c_2}\cdot n^{-1} =: \tilde c'(q)\cdot n^{-1}\label{proof_lemma_bernstein_eq03}
\end{equation}
With \reff{proof_lemma_bernstein_eq03} and Markov's inequality, we obtain
\begin{eqnarray}
	\IP(|W_n| > \gamma) &\le& \IP(|W_n - W_n^{\circ}| > \frac{\gamma}{2}) + \IP(|W_n^{\circ}| > \frac{\gamma}{2})\nonumber\\
	&\le& \tilde c'(1) \cdot \frac{n^{-1}}{\gamma} + \IP(|W_n^{\circ}| > \frac{\gamma}{2}).\label{proof_lemma_bernstein_eq05}
\end{eqnarray}
\emph{Step 2: Applying a Bernstein inequality from \cite{doukhanneumann2007}}. Note that by Assumption \ref{assumption_dependence} and \reff{lemma_emp_eq1}, we have with some constant $\tilde c$ only dependent on $M,\chi,C,D$:
\[
	\delta_q^{g(\tilde Y(u))}(t) \le \tilde c\cdot \Big(\sum_{j=1}^{t}\chi_j\cdot \rho^{t-j+1}\Big)N_{\alpha}(qM)^M \le \tilde c\cdot t\rho^{t+1} N_{\alpha}(qM)^M,
\]
i.e.
\begin{equation}
	\delta_q^{g(\tilde Y(u))}(t) \le \tilde c'\cdot \tilde \rho^{t}\cdot N_{\alpha}(qM)^M\label{proof_lemma_bernstein_eq1}
\end{equation}
with $\tilde c' > 0$, $\tilde \rho \in (0,1)$ only depending on $M,\chi,C,D,\rho$.

For $s_1,...,s_u,t_1,...,t_v\in\IN$, it holds that
\begin{eqnarray*}
	&& |\Cov(Z_{s_1}^{\circ}\dots Z_{s_u}^{\circ},Z_{t_1}^{\circ}\dots Z_{t_v}^{\circ})|\\
	&\le& \sum_{k=0}^{\infty}|P_{s_u-k}(Z_{s_1}^{\circ}\dots Z_{s_u}^{\circ})\cdot P_{s_u-k}(Z_{t_1}^{\circ}\dots Z_{t_v}^{\circ})|\\
	&\le& \sum_{k=0}^{\infty}\|P_{s_u-k}(Z_{s_1}^{\circ}\dots Z_{s_u}^{\circ})\|_2\cdot\|P_{s_u-k}(Z_{t_1}^{\circ}\dots Z_{t_v}^{\circ})\|_2.
\end{eqnarray*}
We have
\[
	\|P_{s_u-k}(Z_{s_1}^{\circ}\dots Z_{s_u}^{\circ})\|_2 \le 2\|Z_{s_1}^{\circ}\dots Z_{s_u}^{\circ}\|_2 \le 2(2|\phi|_{\infty}a_n)^u.
\]
By Lemma \ref{lemma_exponential_moments}, $L := |\phi|_{\infty}\sup_{u\in[0,1]}\|g(\tilde Y_0(u))\|_2 \le |\phi|_{\infty}\cdot \tilde c_1 N_{\alpha}(2M)^M$. We obtain with \reff{proof_lemma_bernstein_eq1} and Lipschitz continuity of $\Psi$ that
\begin{eqnarray*}
	\|P_{s_u-k}(Z_{t_1}^{\circ}\dots Z_{t_v}^{\circ})\|_2 &\le& \|Z_{t_1}^{\circ}\dots Z_{t_v}^{\circ} - (Z_{t_1}^{\circ})^{*}\dots (Z_{t_v}^{\circ})^{*}\|_2\\
	&\le& L|\phi|_{\infty}(2|\phi|_{\infty}a_n)^{u-2}\sum_{k=1}^{v}\delta_2^{g(\tilde Y(u))}(t_1 - s_u + k)\\
	&\le& \tilde c'\cdot L|\phi|_{\infty} (2|\phi|_{\infty}a_n)^{u-2}\cdot v \cdot \tilde \rho^{t_1-s_u}.
\end{eqnarray*}
Furthermore,
\[
	\sum_{s=0}^{\infty}(s+1)^k \tilde \rho^{k} \le (\frac{1}{1-\tilde \rho})^{k+1},\quad\quad \IE|Z_{t}|^k \le (2|K|_{\infty}a_n)^k.
\]
Using Theorem 1 in \cite{doukhanneumann2007} (with $\mu=0$, $\nu=1$ therein) yields 
\[
	\IP(W_n^{\circ} > \frac{\gamma}{2}) \le \IP(\sum_{t=1}^{n}Z_{t}^{\circ} \ge \frac{\gamma}{2}) \le \exp\Big(-\frac{(\gamma/2)^2/2}{A_n + B_n^{1/3}(\gamma/2)^{5/3}}\Big),
\]
where $A_n :=\Var(W_n^{\circ})$ and with $s_n := \#\{t\in \{1,...,n\}:\phi(t/n) \not=0\}$,
\[
	B_n = 2(\sqrt{\tilde c L |\phi|_{\infty}} \vee (2|\phi|_{\infty}a_n))\cdot \frac{1}{1-\tilde \rho}\cdot \Big(\Big(\frac{2^5 s_n \frac{\tilde c' L |\phi|_{\infty}}{1-\tilde \rho}}{A_n}\Big) \vee 1\Big)
\]
(we use $s_n$ instead of $n$ which is possible due to a change in the upper bound in their equation (43)).
Here, we have with \reff{proof_lemma_bernstein_eq03} and $h \le 1$ that
\[
	|\Var(W_n^{\circ}) - \Var(W_n)| \le \|W_n^{\circ} - W_n\|_2 \le \tilde c'(2)n^{-1}.
\]
By assumption, we conclude that $B_n \le \text{const.}(M,\chi,C,D,\rho)\cdot a_n$ for $n$ large enough.
\end{proof}

As a direct corollary of Lemma \ref{lemma_bernstein}, we obtain Theorem \ref{lemma_bernstein_kerne} by using the following arguments:

\begin{proof}[Proof of Theorem \ref{lemma_bernstein_kerne}]
	We apply Lemma \ref{lemma_bernstein} with $\phi(v) := K((v-u)/h)$. Here, $s_n = \#\{t\in \{1,...,n\}: \phi(t/n) \not= 0\} \le n\cdot h$. By Lemma \ref{lemma_variance_calc}, we have $W_n = (nh)\cdot \tilde G_h(u)_j$ and 
	\[
		|\Var(\tilde G_h(u)_j) - v_j^2(h,u)| \le c\cdot ((nh)^{-2} + n^{-1}\log(n)^2),
	\]
	showing that $\frac{1}{2}(nh)^2 v_j^2(h,u) \le \Var(W_n) \le 2(nh)^2 v_j^2(h,u)$ is fulfilled for $n$ large enough since $h\in H_n$. Therefore, 
	\reff{lemma_bernstein_condition1} is fulfilled for $n$ large enough.
	%\begin{eqnarray*}
	%	\|W_n\|_2 &\le& \sum_{k=0}^{\infty}\Big\| \sum_{t=1}^{n}\phi(t/n) P_{t-k}g(\tilde Y_t(t/n))\Big\|_2 = \sum_{k=0}^{\infty}\Big(\sum_{t=1}^{n}\phi(t/n)^2 \|P_{t-k} g(\tilde Y_t(t/n))\|_2^2\Big)^{1/2}\\
	%	&\le& \sum_{k=0}^{\infty}\sup_{u\in[0,1]}\delta_2^{g(\tilde Y(u))}(k) \cdot \Big(\sum_{t=1}^{n}\phi(t/n)^2\Big)^{1/2} \le \frac{\tilde c N_{\alpha}(2M)^M}{1-\tilde \rho}\cdot \Big(\sum_{t=1}^{n}\phi(t/n)^2\Big)^{1/2}.
%\end{eqnarray*}	
\end{proof}

Furthermore, we obtain a Bernstein inequality for the difference $\tilde G_h(u) - \tilde G_{h'}(u)$ for two different bandwidths $h' \le h$:

\begin{lemma}[Bernstein inequality for $\tilde G_h(u)-\tilde G_{h'}(u)$]\label{lemma_bernstein_kerne2}
	Fix $u\in [0,1]$ and $a \in (0,1)$. Assume that $g\in \sH(M,\chi,C)$ and Assumption \ref{assumption_dependence} holds. Then there exist some constants $c_4,c_5 > 0$ only dependent on $M,\chi,C,D,\rho,|K|_{\infty},a$ such that for $h' \le a\cdot h$,
	\begin{eqnarray*}
		&&\IP\Big((nh')\big|(\tilde G_h(u)_j - \tilde G_{h'}(u)_j) - (\IE \tilde G_h(u)_j - \IE \tilde G_{h'}(u)_j)\big| > \gamma\Big)\\
		&\le& 2\exp\Big(-\frac{\gamma^2}{32 (nh')^2v_j^2(h,h',u) + c_4 a_n^{1/3} \gamma^{5/3}}\Big) + c_5 \frac{n^{-1}}{\gamma}.
	\end{eqnarray*}
		%\IP\Big(\big|\sum_{t=1}^{n}\phi(t/n) \{g(\tilde Y_t(t/n)) - \IE g(\tilde Y_t(t/n))\}\big| > \gamma\Big) \le 2\exp\Big(-\frac{\gamma^2}{16(nh)^2 \sigma_{n,h}^2 + c_4 a_n^{1/3} \gamma^{5/3}}\Big) + c_5 \frac{n^{-1}}{\gamma}.
	with $a_n := \tilde c_1(8\log(n))^{1/\tau_2}$ ($\tilde c_1, \tau_2$ from Lemma \ref{lemma_exponential_moments}).
	%\begin{eqnarray*}
		%\sigma_{n,h}^2 &:=& \frac{1}{nh}\int_0^{1}K(x)^2 \dif x \cdot \sum_{k\in\Z}\Cov(g(\tilde Y_0(u)), g(\tilde Y_k(u))),\\
		%\sigma_{n,h,h'}^2 &:=& \frac{1}{n}\int_0^{1}\{K_h(x) - K_{h'}(x)\}^2 \dif x \cdot \sum_{k\in\Z}\Cov(g(\tilde Y_0(u)), g(\tilde Y_k(u))).
	%\end{eqnarray*}
\end{lemma}
\begin{proof}[Proof of Theorem \ref{lemma_bernstein_kerne}]
	We apply Lemma \ref{lemma_bernstein} with $\phi(v) := h'(K_h(v-u) - K_{h'}(v-u))$. Here $\#\{t\in \{1,...,n\}: \phi(t/n) \not= 0\} \le 2n h$. By Lemma \ref{lemma_variance_calc}, we have $W_n = (nh')\cdot \tilde G_h(u)_j$ and 
	\[
		|\Var(\tilde G_h(u)_j - \tilde G_{h'}(u)_j) - v_j^2(h,h',u)| \le c\cdot ((nh')^{-2} + n^{-1}\log(n)^2),
	\]
	showing that
	\[
		\frac{1}{2}(nh')^2 v_j^2(h,h',u) \le \Var(W_n) \le 2(nh')^2 v_j^2(h,h',u)
	\]
	is fulfilled for $n$ large enough since $h\in H_n$. Therefore, 
	\reff{lemma_bernstein_condition1} is fulfilled for $n$ large enough.
	%\begin{eqnarray*}
	%	\|W_n\|_2 &\le& \sum_{k=0}^{\infty}\Big\| \sum_{t=1}^{n}\phi(t/n) P_{t-k}g(\tilde Y_t(t/n))\Big\|_2 = \sum_{k=0}^{\infty}\Big(\sum_{t=1}^{n}\phi(t/n)^2 \|P_{t-k} g(\tilde Y_t(t/n))\|_2^2\Big)^{1/2}\\
	%	&\le& \sum_{k=0}^{\infty}\sup_{u\in[0,1]}\delta_2^{g(\tilde Y(u))}(k) \cdot \Big(\sum_{t=1}^{n}\phi(t/n)^2\Big)^{1/2} \le \frac{\tilde c N_{\alpha}(2M)^M}{1-\tilde \rho}\cdot \Big(\sum_{t=1}^{n}\phi(t/n)^2\Big)^{1/2}.
%\end{eqnarray*}	
\end{proof}

\begin{lemma}[Stationary approximation]\label{lemma_statapprox}
	Let $q \ge 1$. Then with some constant $c' = c'(u) > 0$,
	\begin{enumerate}
	\item[(i)] $\|\hat G_h(u) - \tilde G_h(u)\|_{q} \le c' n^{-1}$.
	\item[(ii)] uniformly in $u\in[0,1]$, $\|\hat G_h(u) - \tilde G_h(u)\|_{q} \le c' (nh)^{-1}$.
	\end{enumerate}
\end{lemma}
\begin{proof}[Proof of Lemma \ref{lemma_statapprox}] (i) By Hoelder's inequality,
\begin{eqnarray}
	&& \|g(Y_{t,n}) - g(\tilde Y_t(t/n))\|_q\nonumber\\
	&\le& C(1+2DN_{\alpha}(qM)^{M-1})\cdot \Big[\sum_{i=1}^{t-1}\chi_i \|X_{t-i,n} - \tilde X_{t-i}((t-i)/n)\|_{qM} + \sum_{i=t}^{\infty}\chi_i \|\tilde X_{t-i}(t/n)\|_{qM}\Big]\nonumber\\
	&\le& CD(1+2DN_{\alpha}(qM)^{M-1})\cdot \Big[n^{-1}\sum_{i\in\N}\chi_i + \sum_{i=t}^{\infty}\chi_i\Big]\label{lemma_statapprox_eq1}
\end{eqnarray}
Put $c'' = CD(1+2DN_{\alpha}(qM)^{M-1})$. Then
	\[
		\|\hat G_h(u) - \tilde G_h(u)\|_{q} \le |K|_{\infty}c''n^{-1}\sum_{i\in\N}\chi_i + \frac{c''}{nh}\sum_{t=1}^{n}K\Big(\frac{t/n-u}{h}\Big)\cdot \sum_{i=t}^{\infty}\chi_i.
	\]
	If $h \le u$, then summation is only done over $\frac{t}{n} \ge u-\frac{h}{2} \ge \frac{u}{2}$, thus $t \ge \frac{u}{2}\cdot n$. Since $\sum_{i=t}^{\infty}\chi_i \le c t^{-1} \le \frac{2c}{u}\cdot \frac{1}{n}$, we obtain the result.\\
	If $h \ge u$, then $|\frac{1}{nh}\sum_{t=1}^{n}K\big(\frac{t/n-u}{h}\big)\cdot \sum_{i=t}^{\infty}\chi_i| \le \frac{|K|_{\infty}}{nh}\sum_{t=1}^{\infty}\sum_{i=t}^{\infty}\chi_i$ and $nh \ge nu$, we obtain  the result.
	
	(ii) is immediate from \reff{lemma_statapprox_eq1} and
	\[
	    \|\hat G_h(u) - \tilde G_h(u)\|_{q} \le \frac{|K|_{\infty}}{nh}\sum_{t=1}^{n}\|g(Y_{t,n}) - g(\tilde Y_t(t/n))\|_q.
	\]
\end{proof}

\begin{lemma}[Calculation of Variance]\label{lemma_variance_calc} Assume that
\[
    \sigma_p^2(u) := \sum_{k\in\Z}\Cov(p(\tilde Y_0(u)), p(\tilde Y_k(u)))
\]
fulfills $\sigma_{p,min}^2 := \inf_{u\in[0,1]}\sigma^2_p(u) > 0$.
For $\tilde P_h(u) = \frac{1}{n}\sum_{t=1}^{n}K_h(t/n-u)\cdot p(\tilde Y_t(t/n))$ with $p \in \sH(M,\chi,C)$, it holds that (i)
\[
	\Var(\tilde P_h(u)) = \frac{1}{nh}\int K(x)^2 \dif x \cdot \sigma_p^2(u) + R_{n,h},
\]
and for $h' \le a\cdot h$ (with fixed $a\in(0,1)$), there exist constants $c',c_6 > 0$ such that (ii)
\begin{eqnarray*}
	\Var(\tilde P_h(u) - \tilde P_{h'}(u)) &\ge& c_6 (a-1)^2\sigma_{p,min}^2\cdot (nh')^{-1},\\
	\Var(\tilde P_h(u) - \tilde P_{h'}(u)) &=& \frac{1}{n}\int \{K_h(x) - K_{h'}(x)\}^2 \dif x \cdot \sigma_p^2(u) + R_{n,h,h'},
\end{eqnarray*}
where $|R_{n,h}| \le c'\cdot ((nh)^{-2} + n^{-1}\log(n)^2)$, $|R_{n,h,h'}| \le c'\cdot ((nh')^{-2} + n^{-1}\log(n)^2)$.
\end{lemma}
\begin{proof}[Proof of Lemma \ref{lemma_variance_calc}]
	(i) Put $\bar P_h(u) := \frac{1}{n}\sum_{t=1}^{n}K_h(t/n-u)\cdot p(\tilde Y_t(u))$. With some constant $\tilde c$ only dependent on $M,\chi,C,D$, we have by \reff{proof_lemma_bernstein_eq1},
	\begin{eqnarray*}
		\|P_{t-k}\{p(\tilde Y_t(t/n)) - p(\tilde Y_t(u))\}\|_2 &\le& 2\|p(\tilde Y_t(t/n)) - p(\tilde Y_t(u))\|_2 \le \tilde c\cdot |t/n-u|,\\
		\|P_{t-k}\{p(\tilde Y_t(t/n)) - p(\tilde Y_t(u))\}\|_2 &\le& 2\sup_{u\in[0,1]}\delta_2^{p(\tilde Y(u))}(k) \le \tilde c\cdot \tilde \rho^{k}
	\end{eqnarray*}
	 and thus
	\begin{eqnarray}
		&& \|(\tilde P_h(u) - \IE \tilde P_h(u)) - (\bar P_h(u) - \IE \bar P_h(u))\|_2\nonumber\\
		&\le& \frac{1}{nh}\sum_{k=0}^{\infty}\Big\|\sum_{t=1}^{n}K((t/n-u)/h)\cdot P_{t-k}\{p(\tilde Y_t(t/n)) - p(\tilde Y_t(u))\}\Big\|_2\nonumber\\
		&=&\frac{1}{nh}\sum_{k=0}^{\infty}\Big(\sum_{t=1}^{n} K((t/n-u)/h)^2\cdot \|P_{t-k}\{p(\tilde Y_t(t/n)) - p(\tilde Y_t(u))\}\|_2^2\Big)^{1/2}\nonumber\\
		&\le& \tilde c \cdot |K|_{\infty}(nh)^{-1/2}\sum_{k=0}^{\infty} \min\{\tilde \rho^{k},h\}.\label{lemma_variance_calc_eq104}
	\end{eqnarray}
	Since $h \ge n^{-1}$, we have with some constant $\bar c$ only dependent on $\tilde \rho$ that
	\begin{equation}
		%\sum_{k=0}^{\infty} \min\{ (\tilde \rho^{1/2})^{k},h^{1/2}\} \le \sum_{k=0}^{\lfloor \log(h)/\log(\tilde \rho)\rfloor}h^{1/2} + \sum_{k=\lceil \log(h)/\log(\tilde \rho)\rceil}^{\infty}(\tilde \rho^{1/2})^k \le \bar c \cdot h^{1/2}\log(n),\label{lemma_variance_calc_eq2}
		\sum_{k=0}^{\infty} \min\{ \tilde \rho^{k},h\} \le \sum_{k=0}^{\lfloor \log(h)/\log(\tilde \rho)\rfloor}h + \sum_{k=\lceil \log(h)/\log(\tilde \rho)\rceil}^{\infty}\tilde \rho^k \le \bar c \cdot h\log(n),\label{lemma_variance_calc_eq2}
	\end{equation}
	thus
	\begin{equation}
		\|(\tilde P_h(u) - \IE \tilde P_h(u)) - (\bar P_h(u) - \IE \bar P_h(u))\|_2 \le \tilde c \bar c |K|_{\infty}n^{-1/2}h^{1/2}\log(n).\label{lemma_variance_calc_eq10}
	\end{equation}
	Abbreviate $p_t := p(\tilde Y_t(u))$. Then
\begin{eqnarray}
	&& \|\bar P_h(u)-\IE \bar P_h(u)\|_2^2\nonumber\\
	&=& \frac{1}{n^2}\Big\|\sum_{k=0}^{\infty}\sum_{t=1}^{n}K_h(t/n-u)\cdot P_{t-k}p_t\Big\|_2^2\nonumber\\
	&=& \frac{1}{n^2}\sum_{k,l=0}^{\infty}\IE \Big[\sum_{t=1}^{n}K_h(t/n-u)\cdot P_{t-k}p_t\cdot \sum_{s=1}^{n}K_h(s/n-u)\cdot P_{s-l}p_s\Big]\nonumber\\
	&=& \frac{1}{(nh)^2}\sum_{k,l=0}^{\infty}\IE[P_{0}p_k\cdot P_{0}p_{l}]\nonumber\\
	&&\quad\quad\quad\quad\times \sum_{t:1\le t \le n, 1 \le t-k+l\le n}K((t/n-u)/h)K(((t-k+l)/n-u)/h).\label{lemma_variance_calc_eq1}
\end{eqnarray}
$K(((t-k+l)/n-u)/h)$ can be replaced by $K((t/n-u)/h)$ due to Lipschitz-continuity (Lipschitz constant $L_K$) of $K$ with replacement error
\begin{eqnarray*}
	&\le& \frac{L_K}{(nh)^3}\sum_{k,l=0}^{\infty}(k+l)\cdot |\IE[P_{0}p_k\cdot P_{0}p_{l}]|\sum_{t=1}^{n}|K((t/n-u)/h)|\\
	&\le& \frac{|K|_{\infty}L_K}{(nh)^2}\sum_{k,l=0}^{\infty} (k+l)\cdot \sup_{u\in[0,1]}\delta_2^{p(\tilde Y(u))}(k) \cdot  \sup_{u\in[0,1]}\delta_2^{p(\tilde Y(u))}(l) \le \tilde C\cdot |K|_{\infty}L_K(nh)^{-2}
\end{eqnarray*}
due to \reff{proof_lemma_bernstein_eq1} with some $\tilde C$ only depending on $M,\chi,C,D,\rho$. After the replacement, \reff{lemma_variance_calc_eq1} reads
\begin{eqnarray*}
	&&\frac{1}{(nh)^2}\sum_{t=1}^{n}K((t/n-u)/h)^2\cdot \sum_{k,l=0}^{\infty}\IE[P_{0}p_k\cdot P_{0}g_{l}]\\
	&=& \frac{1}{(nh)^2}\sum_{t=1}^{n}K((t/n-u)/h)^2\cdot \sigma_p^2(u).
\end{eqnarray*}
Since $K$ is Lipschitz continuous, this can be replaced by $\frac{1}{nh}\int K(x)^2 \dif x \cdot \sigma_p^2(u)$ with replacement error $L_K |K|_{\infty}(nh)^{-2}\sigma_p^2(u)$.\\
(ii) First note that we have with $Q := \frac{h'}{h} \in (0,a]$:
\[
    \frac{1}{n}\int \{K_h(x) - K_{h'}(x)\}^2 \dif x = \frac{1}{nh'}\int \{ Q K(Q y) - K(y) \}^2 \dif y
\]
Let $f(Q) := \frac{1}{(Q-1)^2}\int \{Q K(Qy) - K(y)\}^2 \dif y$. Then
\begin{eqnarray*}
    \lim_{Q \to 1}f(Q) &=& \lim_{Q\to 1}\int \{Qy \frac{K(Qy) - K(y)}{Qy-y} + K(y)\}^2 \dif y\\
    &=& \int \{y K'(y) + K(y)\}^2 \dif y > 0
\end{eqnarray*}
by assumption, and
\[
    \lim_{Q\to0}f(Q) = \int K(y)^2 \dif y > 0.
\]
Since $Q \mapsto f(Q)$ is a continuous function, we conclude that $f_{min} := \inf_{Q\in[0,1]}f(Q) > 0$, and thus
\[
    \alpha_{n,h,h'}^2 := \frac{1}{n}\int\{K_h(x) - K_{h'}(x)\}^2 \dif x \cdot \sigma_{p}^2(u)
\]
fulfills
\[
	\frac{nh'}{(Q-1)^2}\alpha_{n,h,h'}^2 \ge \sigma_{p,min}^2\cdot f_{min} > 0. 
\]
We conclude that
\[
	\alpha_{n,h,h'}^2 \ge (a-1)^2\sigma_{p,min}^2 f_{min} \cdot (nh')^{-1}.
\]
The rest of the proof is the same as in (i).

\end{proof}



%% References


\bibliographystyle{plain}
\bibliography{adaptation}






% \end{thebibliography}

\end{document}